\documentclass[a4paper,10pt, reqno]{amsart}

\usepackage[english]{babel}
\usepackage[utf8]{inputenc}

\usepackage{amssymb}
\usepackage{amsmath, mathtools, mathrsfs,braket,comment}
\usepackage{amsthm}
\usepackage{bbm}
\usepackage{cite}
\usepackage[shortlabels]{enumitem}
\usepackage{tikz, xcolor}
\usepackage{marginnote}

\usepackage{hyperref}
\hypersetup{
    colorlinks=true,
    linkcolor=blue,
    citecolor=red,
    filecolor=magenta,      
    urlcolor=cyan,
    pdftitle={Ultracontractivity and Perturbation},
    linktocpage=true,
}

\usepackage[hyphenbreaks]{breakurl}

\newcommand{\bbC}{\mathbb{C}}

\newcommand{\bbN}{\mathbb{N}}

\newcommand{\bbR}{\mathbb{R}}

\newcommand{\calL}{\mathcal{L}}

\newcommand{\rmE}{\operatorname{E}}
\newcommand{\fra}{\mathfrak{a}}

\newcommand{\suchthat}{\,|\,} 
\DeclareMathOperator{\one}{\mathbbm{1}} 
\DeclareMathOperator{\re}{Re} 
\DeclareMathOperator{\rg}{rg} 
\DeclareMathOperator{\dist}{dist} 
\DeclareMathOperator{\divergence}{div}
\newcommand{\argument}{\mathord{\,\cdot\,}} 
\newcommand{\dd}{\;\mathrm{d}} 
\newcommand{\norm}[1]{\left\lVert #1 \right\rVert} 
\newcommand{\modulus}[1]{\left\lvert #1 \right\rvert} 
\DeclareMathOperator{\dom}{dom} 
\newcommand{\restricted}[1]{|_{#1}}
\newcommand{\Cont}{\mathrm{C}}
\DeclareMathOperator{\Lip}{Lip} 

\newcommand{\spec}{\sigma} 
\newcommand{\spb}{\mathrm{s}} 
\newcommand{\Res}{\mathcal{R}} 

\theoremstyle{definition}
\newtheorem{definition}{Definition}[section]
\newtheorem{remark}[definition]{Remark}
\newtheorem{remarks}[definition]{Remarks}

\newtheorem{example}[definition]{Example}
\newtheorem{examples}[definition]{Examples}

\newtheorem{setting}[definition]{Setting}

\theoremstyle{plain}
\newtheorem{proposition}[definition]{Proposition}
\newtheorem{lemma}[definition]{Lemma}
\newtheorem{theorem}[definition]{Theorem}

\numberwithin{equation}{section}

\title[Smoothing and relatively bounded perturbations]{Smoothing of operator semigroups under relatively bounded perturbations}

\author{Sahiba Arora}
\address[Sahiba Arora]{Leibniz Universität Hannover, Institut für Analysis, Welfengarten 1, 30167 Hannover, Germany}
\email{sahiba.arora@math.uni-hannover.de}

\author{Jonathan Mui}
\address[Jonathan Mui]{University of Wuppertal, School of Mathematics and Natural Sciences, Gaußstr.\ 20, 42119 Wuppertal, Germany}
\email{jomui@uni-wuppertal.de}

\date{\today}



\begin{document}

\begin{abstract}
    We investigate a smoothing property for strongly-continuous operator semigroups, akin to ultracontractivity in parabolic evolution equations. Specifically, we establish the stability of this property under certain relatively bounded perturbations of the semigroup generator. This result yields a spectral perturbation theorem, which has implications for the long-term dynamics of evolution equations driven by elliptic operators of second and higher orders. In particular, a new perturbation theorem for so-called eventually positive semigroups is derived as a consequence of the general results.
\end{abstract}

\maketitle

\section{Introduction}

\subsection{Main results}

To motivate the main ideas of the paper, consider the Cauchy problem for the heat equation
\[
    \begin{aligned}
        \frac{\partial u}{\partial t} &= \Delta u \quad&&\text{on } (0,\infty)\times\bbR^n \\
        u(\argument,0) &= u_0 \quad &&\text{on }\bbR^n
    \end{aligned}
\]
with initial datum $u_0\in L^p(\bbR^n)$, $1\le p < \infty$. It is well known that 
the solution of the above PDE can be expressed by
\[
    u(t,x) = \int_{\bbR^n} K(t,x,y)u_0(y)\dd y, \quad (t,x)\in (0,\infty)\times\bbR^n;
\]
for an integral kernel $K(t,\argument,\argument)\in L^\infty(\bbR^n \times\bbR^n)$ satisfying an estimate of the form
\[
    \modulus{K(t,x,y)}\le Me^{\omega t} t^{-n/2}, \quad t>0, \; (x,y)\in\bbR^n\times\bbR^n
\]
where $M\ge 1$ and $\omega \in \bbR$. It is a classical part of PDE theory that such kernel estimates still hold when the Laplacian is replaced by a general second-order uniformly elliptic operator with bounded coefficients. The key functional-analytic tools to obtain these kernel estimates are the Dunford-Pettis criterion and the \emph{ultracontractivity} of the solution semigroup $(T(t))_{t\ge 0}$, namely the smoothing estimate
\[
    \norm{T(t)}_{\calL(L^1, L^\infty)} \le Me^{\omega t}t^{-n/2}, \quad t>0.
\]
This is a distinguishing feature of diffusion-type equations on Euclidean domains and Riemannian manifolds and has been studied extensively due to its close connections with important functional inequalities (e.g.\ Nash and Sobolev). As such, ultracontractivity appears as an important analytical tool for applications in spectral theory and geometry; see, for instance, \cite[Chapters 4 \& 5]{Davies1990} and~\cite[Chapter 6]{BGL}.

From the perspective of semigroup theory, it is natural to consider perturbations of the generator and to ask: which properties of the unperturbed semigroup carry over to the perturbed semigroup; for classical results in this direction, see \cite[Chapter~III]{EngelNagel} and \cite[Chapter~IX]{Kato}. For ultracontractivity, perturbations by singular potentials were investigated by Davies in~\cite{Davies1986-perturbation} using quadratic form methods and logarithmic Sobolev inequalities; see also \cite[Section~4.8]{Davies1990}, which contains applications to Schrödinger semigroups. Whereas Davies' treatment is restricted to semigroups generated by second-order elliptic operators (specifically, \emph{symmetric Markovian} semigroups), the present paper investigates the preservation of ultracontractivity for semigroups under unbounded perturbations in a significantly general setting (Theorem~\ref{thm:perturb-smoothing}). As an appetiser, we state the version for bounded perturbations, which is proved at the end of Section~\ref{sec:theorem-ultracontractive}.

\begin{theorem}[Ultracontractivity under bounded perturbations]
    \label{thm:ultracontractivity-introduction}
    Let $(T(t))_{t\ge 0}$ be a $C_0$-semigroup on a Banach space $X$ with generator $A$ such that $(T(t))_{t\ge 0}$ is \emph{ultracontractive with respect to a Banach space $V$}, i.e.,  $V$ embeds continuously into $X$ and there exist $C\ge 1$ and $\alpha>0$ such that
    \[
        T(t)X\subseteq V \quad \text{and}\quad \norm{T(t)}_{X\to V} \le Ct^{-\alpha} \quad \text{for all } t\in (0,1].
    \]
    Moreover, assume that $\norm{T(t)}_{V\to V} \le C$ for all $t\in (0,1]$.

    If $B:X\to X$ is a bounded operator leaving $V$ invariant, then the $C_0$-semigroup $(S(t))_{t\ge 0}$ generated by $A+B$ is also ultracontractive with respect to $V$, i.e., 
    \[
        S(t)X\subseteq V \quad \text{and}\quad \norm{S(t)}_{X\to V} \le \widetilde{C}t^{-\alpha} 
    \]
     for all $t\in (0,1]$ and for some $\widetilde{C}\ge 1$.
\end{theorem}

The second focus of the paper is the spectral behaviour of ultracontractive semigroups under perturbations. Let us assume that the operators $A$ and $B$ on a Banach space $X$ fulfil the conditions stated in Theorem~\ref{thm:ultracontractivity-introduction}. If $\lambda_0$ is a pole of the resolvent $\Res(\argument,A)$, then it is well known -- even under more general conditions -- that for $\kappa$ in some $U\subseteq \bbC$, the operator $A+\kappa B$ also has a spectral value $\lambda(\kappa)$ that is a pole of the resolvent $\Res(\argument,A+\kappa B)$. Furthermore, if $P(\kappa)$ denotes the spectral projection of $\lambda(\kappa)$ associated to $A+\kappa B$, then the mapping
\[
    U \ni \kappa \mapsto P(\kappa) \in \calL(X)
\]
is analytic. However, it turns out that our assumptions even imply the analyticity of the above mapping with co-domain $\calL(V)$ (see Theorem~\ref{thm:analyticity-spectrum}). Furthermore, if $X$ has a nice order structure, then we are able to preserve a lower bound on eigenfunctions. This is made precise in the following theorem, which is a particular case of Theorem~\ref{thm:eigenfunction-lower-bound} below. The relevant terminology on Banach lattices is recalled in Section~\ref{sec:notation}.

\begin{theorem}
    \label{thm:analyticity-introduction}
    Let $(T(t))_{t\ge 0}$ be a $C_0$-semigroup on a complex Banach lattice $E$ whose generator $A$ is real. Let $0 \le u\in E$ be such that $(T(t))_{t\ge 0}$ is ultracontractive with respect to the principal ideal $E_u$.  Moreover, assume that there exists $C\ge 1$ such that $\norm{T(t)}_{E_u\to E_u} \le C$ for all $t\in (0,1]$ and let $B:E\to E$ be a bounded operator leaving $E_u$ invariant.
    
    If $\lambda_0 \in \spec(A)\cap \bbR$ is a pole of the resolvent $\Res(\argument, A)$ and an algebraically simple eigenvalue with an eigenvector satisfying $v_0\ge cu$ for some constant $c>0$, then there exists $\delta>0$ such that for all $\kappa \in (-\delta,\delta)$, the operator $A+\kappa B$ has an algebraically simple eigenvalue $\lambda(\kappa)$ with an eigenvector $v(\kappa)\ge (c/2) u$.
\end{theorem}

The study of lower bounds on eigenfunctions in this abstract setting was motivated by recent works on \emph{eventually positive semigroups}; see~\cite{Glueck2022} for a survey. While semigroups arising from diffusion equations are \emph{positive}, in the sense that positive initial data $u_0 \ge 0$ yield positive solutions $u(t;u_0) \ge 0$ for all $t>0$, such a property fails if one considers semigroups arising from higher-order elliptic differential operators, or second-order operators with certain `exotic' (e.g.\ non-local) boundary conditions. Nevertheless, in these situations, one may still ask whether solutions are eventually positive, i.e.\ if there exists $t_0 \ge 0$ such that $u(t;u_0) \ge 0$ for all $t\ge t_0$.

A functional analytic framework to analyse this property was developed in~\cite{DanersGlueckKennedy2016a, DanersGlueckKennedy2016b, DanersGlueck2018b}, where it was also applied to a variety of PDE examples. In many cases, the smoothing property in Theorem~\ref{thm:ultracontractivity-introduction} and the lower bound on eigenfunctions in Theorem~\ref{thm:analyticity-introduction} are sufficient conditions used to establish the so-called `eventual strong positivity' of semigroups on Banach lattices. Thus, the stability of the aforementioned conditions under perturbations has direct implications for the perturbation theory of eventually positive semigroups, which until now is fairly limited.
Previous work~\cite{DanersGlueck2018a} has been restricted to bounded positive perturbations. Recently, the authors of~\cite{Pappu-etal} initiated a study of unbounded positive perturbations of eventually positive semigroups arising from delay differential equations. The assumptions in that article, however, are quite restrictive. In contrast, we are able to treat non-positive bounded perturbations and a much larger class of unbounded perturbations, which has a natural relationship to PDE problems. In this way, Section~\ref{sec:eventual-positivity} of the present work represents an important step forward in the perturbation theory of eventually positive semigroups.

\subsection{Organisation of the article}

In the remainder of this section, we recall some foundational concepts and fix our notation. The subsequent Section~\ref{sec:ultracontractive} is devoted to the proof of the main perturbation result Theorem~\ref{thm:perturb-smoothing}, from which the assertions of Theorem~\ref{thm:ultracontractivity-introduction} follow. The proof of Theorem~\ref{thm:perturb-smoothing} uses a technical result about Bochner integrals that is contained in Appendix~\ref{appendix:bochner}.
We then investigate how the spectrum of an ultracontractive semigroup behaves under analytic perturbations in Section~\ref{sec:analytic}, in particular, providing a proof of Theorem~\ref{thm:analyticity-introduction}. While perturbations of the spectrum have been extensively studied by Kato in \cite{Kato}, for the convenience of the reader and to provide a precise formulation tailored to our needs, we summarise in Appendix~\ref{appendix:spectrum} how the spectrum behaves under relatively bounded perturbations. Applications of our results to elliptic operators and eventually positive semigroups are presented in Sections~\ref{sec:applications} and~\ref{sec:eventual-positivity} respectively.

\subsection{Notation}
    \label{sec:notation}
Throughout the paper, we assume the reader is familiar with the theory of $C_0$-semigroups; see \cite{EngelNagel, Pazy}. The space of bounded linear operators between Banach spaces $X$ and $Y$ is denoted by $\calL(X,Y)$ with the shorthand $\calL(X)\coloneqq\calL(X,X)$. The range of an operator $T\in \calL(X,Y)$ is denoted by $\rg T$. For a closed operator $A : \dom(A)\subset X\to X$ on a Banach space $X$, we write $A':\dom(A')\subset X'\to X'$ for the Banach space dual of $A$ whenever $A$ is densely defined.
Moreover, the spectrum and resolvent set of $A$ are denoted by $\spec(A)$ and $\rho(A)\coloneqq \bbC \setminus \spec(A)$ respectively. Also, $\Res(\lambda,A)\coloneqq (\lambda- A)^{-1}$ denotes the \emph{resolvent of $A$} at a point $\lambda \in \rho(A)$. 
A spectral value $\lambda_0 \in \spec(A)$ is called a \emph{pole of the resolvent} $\Res(\argument,A)$ if the analytic mapping $\lambda \mapsto \Res(\lambda,A)$ has a pole at $\lambda_0$.

If $\lambda_0$ is a pole of the resolvent $\Res(\argument,A)$, then it is also an eigenvalue of $A$. The eigenvalue $\lambda_0$ is called \emph{algebraically simple} if the dimension of the \emph{generalised eigenspace} associated to $\lambda_0$,
$
    \bigcup_{m\in\bbN} \ker[(\lambda_0-A)^m]
$
is one. Note that if $P$ denotes the spectral projection of $A$ associated to the pole $\lambda_0$, then $\rg P$ coincides with the generalised eigenspace of $\lambda_0$.

When studying the long-term behaviour of the semigroup, we freely use the theory of Banach lattices, for which we refer to the classical monographs~\cite{Schaefer, Meyer-Nieberg1991}. 
Recall that a \emph{complex Banach lattice} is by definition the complexification of a Banach lattice over $\bbR$. We denote the positive cone and the real part of a complex Banach lattice $E$ by $E_+$ and $E_{\bbR}$ respectively.
An operator $A:\dom(A)\subseteq E\to E$ is said to be \emph{real} if 
\[
	\dom(A) = \dom(A) \cap E_{\bbR}+ i \dom(A) \cap E_{\bbR} \quad \text{and} \quad A(\dom(A) \cap E_{\bbR}) \subseteq E_{\bbR}.
\]
For $u \in E_+$, the \emph{principal ideal generated by $u$} is the Banach space
\[
\label{eq:Eu-definition}
    E_u \coloneqq \{ x\in E ~\colon \text{there exists }c>0 \text{ with }\modulus{x} \le cu\}
\]
equipped with the \emph{gauge norm}
\[
    \norm{\argument}_{E_u}\coloneqq \inf\{ c>0 : \modulus{\argument} \le cu\}.
\]
The space $E_u$ always embeds continuously into $E$ and if this embedding is even dense, then we call $u$ a \emph{quasi-interior point of $E_+$}. 

Let $p\in [1,\infty]$ and let $E=L^p(\Omega,\mu)$ for some $\sigma$-finite measure space $(\Omega,\mu)$. Then $u \in E_+$ if and only if $u(\omega)\ge 0$ almost everywhere. If $p<\infty$, then $u$ is a quasi-interior point of $E_+$ if and only if one has $u(\omega)>0$ almost everywhere. For $p=\infty$, the quasi-interior points are exactly those that satisfy $u\ge c\one_{\Omega}$ for some $c>0$. Analogously, if $E=\Cont(K)$ for some compact space $K$, then $u\in E_+$ if and only if $u(\omega)\ge 0$ for all $\omega \in K$ and it is a quasi-interior point of $E_+$ if and only if $u(\omega)>0$ for every $\omega \in K$. In this case, we even have $E=E_u$ and the quasi-interior points are precisely the interior points of the positive cone $E_+$. Note that (for instance) in infinite-dimensional $L^p$ spaces for $1\le p < \infty$, the positive cone has empty interior, which shows that quasi-interior points are the more natural concept in general Banach lattices.

Additionally, for a set $A$, the function $\one_A$ denotes the constant function with value one on $A$.
Further notation and terminology are introduced as needed.

\section{Robustness of ultracontractivity under unbounded perturbations}
    \label{sec:ultracontractive}

In this section, we investigate sufficient conditions for preserving the smoothing property of semigroups under the so-called \emph{Miyadera--Voigt} perturbations.  We begin by establishing the foundational framework in Section~\ref{sec:setting-ultracontractive}, introducing and elaborating upon the pertinent conditions. Thereafter, the sufficiency of these conditions is demonstrated in Section~\ref{sec:theorem-ultracontractive}.

\subsection{Setting up}
    \label{sec:setting-ultracontractive}

Throughout the article, we consider the following setting. For a closed operator $A:\dom(A)\subseteq X\to X$, the interpolation space $X_1$ is the Banach space $\dom(A)$ equipped with the graph norm.

\begin{setting}
    \label{setting:admissible}
    Let $(T(t))_{t\ge 0}$ be a $C_0$-semigroup on a Banach space $X$ with generator $A:\dom(A)\subseteq X\to X$, let $V$ be a Banach space that embeds continuously into $X$, and let $B:\dom(B)\subseteq X\to X$ be an operator such that $\dom(A) \subseteq \dom(B)$ and the restriction
    $B\vert_{X_1} \in\calL(X_1,X)$. In the sequel, we consider the following conditions:
    \begin{enumerate}[\upshape (S1)]
        \item\label{assume:it:ultracontractive}\emph{Ultracontractivity}: There exist constants $C\ge 1$ and $\alpha>0$ such that
        \begin{equation}
        \label{assume:eq:ultracontractive}
            T(t)X\subseteq V \quad\text{and}\quad\|T(t)\|_{X\to V}\le Ct^{-\alpha}
        \end{equation}
        for all $t\in (0,1]$.
    
        \item\label{assume:it:admissibleX} For some $q\in (0,1)$ and $t_0 > 0$, it holds that
        \begin{equation}
        \label{assume:eq:zero-class}
        \int_0^{t_0} \|BT(s)x\|_X \dd s \le q\|x\|_X \qquad\text{for all }x\in \dom(A).
        \end{equation}
        
        \item\label{assume:it:Vbounds} There exist constants $C\ge 1$ and $\beta\in [0,1)$ such that
        \begin{align}
            &\|T(t)\|_{V\to V} \le C \text{ for all }t\in (0,1] \label{assume:eq:Vbound-1},\\
            &BT(t)x\in V \quad\text{and} \quad\norm{BT(t)x}_V \le Ct^{-\beta}\norm{x}_V \label{assume:eq:Vbound-2}
        \end{align}
        whenever $x\in V$ and $t\in (0,1]$ is such that $T(t)x\in \dom(B)$.
    \end{enumerate}
    Without loss of generality, the constant $C\ge 1$ is chosen to be the same in all the estimates~\eqref{assume:eq:ultracontractive}--\eqref{assume:eq:Vbound-2}.
    For the Banach space $V$, we consider one of the following:
    \begin{enumerate}[\upshape (B1)]
        \item\label{assume:bochner-V:itm:separable} 
        The Banach space $V$ is separable. 

        \item\label{assume:bochner-V:itm:closed-ball} 
        The closed unit ball of $V$ is closed in $X$.
    \end{enumerate}
\end{setting}

A few general remarks about the conditions~\ref{assume:it:ultracontractive}--\ref{assume:it:Vbounds} introduced in Setting~\ref{setting:admissible} are collected below, and examples of situations where the condition~\ref{assume:bochner-V:itm:closed-ball} is fulfilled are provided in Appendix~\ref{appendix:bochner}.

\begin{remarks}
    \label{rem:setting}
    (a) Setting~\ref{setting:admissible} admits bounded operators $B\in\mathcal{L}(X)$ which leave $V$ invariant. Indeed, such a $B$ restricts to a bounded operator on $V$ by the closed graph theorem, and thus the compatibility estimate~\eqref{assume:eq:Vbound-2} holds with $\beta=0$ and for all $x\in V$ and all $t\in (0,1]$. In addition,~\eqref{assume:eq:zero-class} is clearly satisfied if $t_0$ is chosen sufficiently small.

    (b) Condition~\ref{assume:it:admissibleX} is particularly relevant in systems and control theory:~given a Banach space $Y$ and a $C_0$-semigroup $(T(t))_{t\ge 0}$ on a Banach space $X$, an \emph{observation operator} $B\in \calL(X_1,Y)$ is called \emph{zero-class $L^1$-admissible} if there exists a function $\tilde{q} : [0,\infty) \to [0,\infty)$ such that $\lim_{t \downarrow 0} \tilde q(t)=0$ and
    \[
        \int_{0}^t \norm{BT(s)x}_Y \dd s \le \tilde q(t) \norm{x}_X \qquad \big(x \in \dom(A)\big).
    \]
    In particular, zero-class $L^1$-admissibility implies~\eqref{assume:eq:zero-class}, which in turn
    ensures that $B$ is $A$-bounded (see Appendix~\ref{appendix:spectrum} for a definition) with $A$-bound strictly less than $1$; see \cite[Exercise~III-2.18(2)]{EngelNagel}. Sufficient conditions for $L^1$-admissibility of observation operators using order properties were recently given in \cite{AroraGlueckPaunonenSchwenninger2025}, see, in particular, \cite[Corollary~5.3 and Appendix~A]{AroraGlueckPaunonenSchwenninger2025}. Its relation with the dual property -- zero-class admissibility of control operators -- can be found in \cite[Theorem~6.9]{Weiss1989a} and \cite{AroraSchwenninger2025}.

    (c) We have called condition~\ref{assume:it:ultracontractive} `ultracontractivity' since it is an abstraction of the classical notion as seen in, for instance, the well-known monograph of Davies~\cite[Chapter 2]{Davies1990} where $X=L^2(\Omega)$ for a bounded open set $\Omega\subseteq\bbR^n$ and $V=L^\infty(\Omega)\hookrightarrow X$. Then in many cases, e.g.\ when $(T(t))_{t\ge 0}$ is a symmetric Markovian semigroup, the exponent in~\eqref{assume:eq:ultracontractive} is explicitly known to be $\alpha = \frac{\mu}{4}$, where $\mu>0$ is the parameter appearing in the Sobolev-type inequality
    \begin{equation*}
        \|u\|_{L^{2\mu/(\mu-2)}(\Omega)} \le C \|(-A)^{1/2}u\|_{L^2(\Omega)} \quad\text{for all } u\in \dom((-A)^{1/2}).
    \end{equation*}
    An overview of key results in the classical theory can be found in~\cite[Section 7]{Arendt2004-survey}. A characterisation of this property in terms of the domain of the generator was recently given in \cite[Theorem~4.1]{AroraGlueck23}.
\end{remarks}

Condition~\ref{assume:it:admissibleX} of Setting~\ref{setting:admissible} guarantees that $B$ is a \emph{Miyadera--Voigt} perturbation~\cite[p.\ 196]{EngelNagel}; cf. Remark~\ref{rem:perturb-smoothing}(b). It is therefore the key to ensure that $A+B$ is also a semigroup generator and that the perturbed semigroup can be represented with the variation of parameters formula. Being a fundamental tool for the main results in this section, we record this and other important facts in the proposition below. The proof can be found,  for instance, in~\cite[Corollary~III.3.15]{EngelNagel}.
\begin{proposition}
    \label{prop:perturbed-semigroup}
    In Setting~\ref{setting:admissible}, assume that $B$ satisfies condition~{\upshape\ref{assume:it:admissibleX}}. Then $A+B$ defined on $\dom(A+B)\coloneqq \dom(A)$ generates a strongly continuous semigroup $(S(t))_{t\ge 0}$ on $X$ that satisfies the \emph{Dyson-Phillips series}
    \begin{equation}
        \label{eq:dyson-phillips}
        S(t) = \sum_{k=0}^\infty S_k(t), \qquad S_{k+1}(t)x \coloneqq \int_0^t S_k(t-s)BT(s)x\dd s
    \end{equation}
    with $S_0(t):=T(t)$
    for all $t\ge 0$ and $x\in \dom(A)$, where the series converges in $\calL(X)$.
    In addition, $S(t)$ can be represented by the \emph{variation of parameters} formula
    \begin{equation}
        \label{eq:variation-param}
        S(t)x = T(t)x + \int_0^t S(t-s)BT(s)x\dd s
    \end{equation}
    for all $x\in \dom(A)$ and $t\ge 0$.
\end{proposition}

\begin{remark}
    \label{rem:variation-of-parameters}
    In Proposition~\ref{prop:perturbed-semigroup}, if $(B, \dom(B))$ is closed, then by \cite[Corollary~III.3.16]{EngelNagel} for every $x\in X$, it holds that $T(t)x, S(t)x\in \dom(B)$ for a.e.\ $t\ge 0$ and the formulae~\eqref{eq:dyson-phillips} and~\eqref{eq:variation-param} extend to $x\in X$.
\end{remark}

\subsection{Preserving ultracontractivity under unbounded perturbations}
    \label{sec:theorem-ultracontractive}

The main result of this section states that the ultracontractivity condition is stable under perturbation by the class of unbounded operators introduced in Setting~\ref{setting:admissible}.
\begin{theorem}
    \label{thm:perturb-smoothing}
    In Setting~\ref{setting:admissible}, assume that conditions~\ref{assume:it:ultracontractive}--\ref{assume:it:Vbounds} hold and that $V$ satisfies at least one of the conditions~\ref{assume:bochner-V:itm:separable} and~\ref{assume:bochner-V:itm:closed-ball}. 
    
    If $(T(t))_{t\ge 0}$ is immediately differentiable, then there exists a constant $\widetilde{C}>0$ such that the perturbed semigroup $(S(t))_{t\ge 0}$ generated by $A+B$ satisfies
    \begin{equation}
    \label{eq:thm:perturb-smoothing}
        S(t)X\subseteq V \quad\text{and}\quad\|S(t)\|_{X\to V}\le \widetilde Ct^{-\alpha}
    \end{equation}
    for all $t\in (0,1]$.
\end{theorem}

The proof of this theorem is somewhat technical. We isolate a part of the proof in Lemma~\ref{lem:V-bounded} which shows that the perturbed semigroup is also well-behaved on $V$. In the rest of this section, the constant $C$ is always the same constant as in Setting~\ref{setting:admissible}. 
Recall that the \emph{Mittag-Leffler function} $\rmE_\alpha$ with parameter $\alpha>0$ is defined by its series expansion
\[
    \rmE_\alpha(z) \coloneqq \sum_{k=0}^\infty \frac{z^k}{\Gamma(k\alpha + 1)};
\]
where $\Gamma(\argument)$ denotes the usual \emph{Gamma function}.
Clearly, $\rmE_1$ coincides with the exponential function $z\mapsto e^z$. Moreover, one can show that the series converges for all $z\in\bbC$ (and thus $E_\alpha$ is an entire function).

\begin{lemma}
    \label{lem:V-bounded}
    Assume that the conditions~\ref{assume:it:admissibleX} and~\ref{assume:it:Vbounds} in Setting~\ref{setting:admissible} hold and that $V$ satisfies at least one of the conditions~\ref{assume:bochner-V:itm:separable} and~\ref{assume:bochner-V:itm:closed-ball}. In addition, suppose that either $(T(t))_{t\ge 0}$ is immediately differentiable or $(B, \dom(B))$ is closed.
    
    Then the perturbed semigroup also $(S(t))_{t\ge 0}$ leaves $V$ invariant and satisfies
    \[
        \|S(t)\|_{V\to V} \le C \rmE_{1-\beta} \left(C\Gamma(1-\beta)t^{1-\beta}\right) \quad\text{for all } t\in [0,1].
    \]
\end{lemma}

We point out that Lemma~\ref{lem:V-bounded} does not impose any ultracontractivity assumptions on the unperturbed semigroup.
The proof of the lemma uses the following observation:
    \begin{equation}
    \label{eq:integral-convolution}
        \int_0^t (t-s)^{\alpha-1} s^{\beta-1} ds 
        = t^{\alpha + \beta - 1} \frac{\Gamma(\alpha) \Gamma(\beta)}{\Gamma(\alpha+\beta)}.
    \end{equation}
for all $\alpha, \beta > 0$ and all $t > 0$.
This can be verified by direct calculation, or alternatively in a more functional-analytic manner as follows: since the left-hand side is the convolution of the functions $(t \mapsto t^{\alpha-1})$ and $(t \mapsto t^{\beta-1})$, its Laplace transform is the product of the Laplace transforms of the functions, i.e., 
  $
        (\argument)^{-(\alpha+\beta)}\Gamma(\alpha) \Gamma(\beta).
  $
Taking inverse Laplace transforms yields the claim.

\begin{proof}[Proof of Lemma~\ref{lem:V-bounded}]
    Since either $(T(t))_{t\ge 0}$ is immediately differentiable or the operator $(B, \dom(B))$ is closed, the compatibility condition~\eqref{assume:eq:Vbound-2} holds for all $x\in V$ and a.e. $t\ge 0$; cf. Remark~\ref{rem:variation-of-parameters}. Moreover, by Proposition~\ref{prop:perturbed-semigroup} and Remark~\ref{rem:variation-of-parameters}, the perturbed semigroup is given by the Dyson-Phillips series
    \begin{equation}
        \label{eq:Dyson-Phillip-proof}
        S(t) = \sum_{k=0}^\infty S_k(t), \qquad S_{k+1}(t)x \coloneqq \int_0^t S_k(t-s)BT(s)x\dd s,
    \end{equation}
    with $S_0(t)=T(t)$ for all $t\ge 0$ and $x\in X$; where the series converges in $\calL(X)$ and the integral converges in the Bochner sense in $X$. We show that for all $t\in [0,1]$, the series $S(t) = \sum_{k=0}^\infty S_k(t)$ converges absolutely in the operator norm of $V$.

    Let us establish inductively for all non-negative integers $k$ that, for all $t\ge 0$
    \begin{equation}
    \label{eq:V-bound-induction}
        S_k(t)V \subseteq V\quad\text{and}\quad \|S_k(t)\|_{V\to V} \le \frac{C^{k+1}\Gamma(1-\beta)^k}{\Gamma(k(1-\beta)+1)}t^{k(1-\beta)}.
    \end{equation}
    For $k=0$, the claim reads $\|S_0(t)\|_{V\to V} \le C$ for all $t\in [0,1]$, which is the assumption~\eqref{assume:eq:Vbound-1} because $S_0(t) = T(t)$. Assume now that~\eqref{eq:V-bound-induction} holds for some integer $k\ge 0$ and
    let $x\in V$. Using the compatibility condition~\eqref{assume:eq:Vbound-2} and the induction hypothesis, we can estimate
    \begin{equation}
        \label{eq:induction-hypothesis-intermediate}
        \|S_k(t-s)BT(s)x\|_{V} \le \frac{C^{k+2}\Gamma(1-\beta)^k}{\Gamma(k(1-\beta)+1)} (t-s)^{k(1-\beta)} s^{-\beta} \norm{x}_V
    \end{equation}
    for a.e.\ $s\in (0,t)$. As the term on the left is integrable and $V$ fulfils at least one out of the conditions~\ref{assume:bochner-V:itm:separable} and~\ref{assume:bochner-V:itm:closed-ball}, Proposition~\ref{prop:bochner-V} together with the recurrence relation~\eqref{eq:Dyson-Phillip-proof} show that $S_{k+1}(t)x \in V$ and satisfies the norm estimate
    \begin{align*}
        \norm{S_{k+1}(t)x}_V & \le \frac{C^{k+2}\Gamma(1-\beta)^k}{\Gamma(k(1-\beta)+1)}\int_0^t (t-s)^{k(1-\beta)} s^{-\beta} \dd s\norm{x}_V\\
                            & = \frac{C^{k+2}\Gamma(1-\beta)^{k+1}}{\Gamma((k+1)(1-\beta)+1)} t^{(k+1)(1-\beta)}\norm{x}_V;
    \end{align*}
    where the equality is obtained using
    \[
        \int_0^t (t-s)^{k(1-\beta)}s^{-\beta}\dd s = 
            t^{(k+1)(1-\beta)} \frac{\Gamma\big(k(1-\beta)+1 \big) \Gamma(1-\beta)}{\Gamma\big( (k+1)(1-\beta) +1 \big)}
    \]
    which holds due to formula~\eqref{eq:integral-convolution}.
    This completes the induction step.

    Lastly, thanks to~\eqref{eq:Dyson-Phillip-proof} and~\eqref{eq:V-bound-induction}, for $t\in (0,1]$ we can now estimate
    \begin{align*}
        \|S(t)\|_{V\to V} 
            \le C \sum_{k=0}^\infty \frac{(C\Gamma(1-\beta)t^{1-\beta})^k}{\Gamma(k(1-\beta)+1)} 
            = \rmE_{1-\beta}\big(C\Gamma(1-\beta)t^{1-\beta}\big).
    \end{align*}
    For $t>1$, the invariance of $V$ under $S$ follows from the semigroup law. 
\end{proof}

The next step in the proof of Theorem~\ref{thm:perturb-smoothing} is an interpolation argument. We refer to~\cite{Lunardi2018} for standard terminology on interpolation methods for Banach spaces.

\begin{lemma}
    \label{lem:interpolation-argument}
    In Setting~\ref{setting:admissible}, assume that conditions~\ref{assume:it:ultracontractive}--\ref{assume:it:Vbounds} hold and that $V$ satisfies at least one of the conditions~\ref{assume:bochner-V:itm:separable} and~\ref{assume:bochner-V:itm:closed-ball}. In addition, suppose that either $(T(t))_{t\ge 0}$ is immediately differentiable or $(B, \dom(B))$ is closed.

    Let $\theta\in (0,1)$ be such that $\alpha\theta < 1-\beta$, let $\mathfrak{F}_{1-\theta}$ be an interpolation method of exponent $1-\theta$, and consider the interpolation space $V_\theta \coloneqq \mathfrak{F}_{1-\theta}(X,V)$. Then the perturbed semigroup $(S(t))_{t\ge 0}$ maps $V_\theta \cap \dom(A)$ into $V$ and
    \begin{equation}
    \label{eq:S-Vtheta-V}
        \|S(t)x\|_{V} \le C_1 t^{-\alpha\theta} \|x\|_{V_\theta} \quad\text{for all } x\in V_\theta \cap \dom(A),\quad t\in (0,1]
    \end{equation}
    for  $C_1\coloneqq C\left(1+\frac{2M_0 C}{1-(\alpha\theta+\beta)}\right)$ and some $M_0\ge 1$.
\end{lemma}

\begin{proof}
    First of all, Lemma~\ref{lem:V-bounded} implies that there exists a constant $M_0\ge 1$ so that
    \begin{equation}
    \label{eq:S-bounded-on-V}
        \sup_{t\in [0,1]}\|S(t)\|_{V\to V} \le M_0.
    \end{equation}
    Also, interpolation between the estimates on $(T(t))_{t\ge 0}$ in~\eqref{assume:eq:ultracontractive} and~\eqref{assume:eq:Vbound-1} yields
    \begin{equation}
        \label{eq:interpolation-inequality-semigoup}
        \|T(t)\|_{V_\theta\to V} \le C^\theta t^{-\alpha\theta} C^{1-\theta} = Ct^{-\alpha\theta}
    \end{equation}
    for $t\in (0,1]$; see, for instance, \cite[Theorem~2.6]{Lunardi2018}. 
    
    Next, let $x\in V_\theta \cap \dom(A)$ and fix $t\in (0,1)$. Since the unperturbed semigroup leaves $X_1$ invariant, the function
    \begin{equation*}
        f_t:(0,t)\to X, \qquad s\mapsto S(t-s)BT(s)x
    \end{equation*}
    is continuous.
    Furthermore, because the unperturbed semigroup maps into $V$ and $x\in \dom(A)$, the compatibility assumption~\eqref{assume:eq:Vbound-2} ensures that
    \[
        \norm{BT(s)x}_V = \|BT(s/2)T(s/2)x\|_V \le C2^{\beta}s^{-\beta}\norm{T(s/2)x}_V
    \]
    for all $s\in (0,t]$.
    Combining this with the uniform bound~\eqref{eq:S-bounded-on-V} and the interpolation property~\eqref{eq:interpolation-inequality-semigoup}, we infer that
    \begin{align*}
        \|f_t(s)\|_V &\le \norm{S(t-s)}_{V\to V}\norm{BT(s)x}_V\\
        & \le M_0 C2^\beta s^{-\beta} \|T(s/2)\|_{V_\theta\to V}\|x\|_{V_\theta} \\
        &\le M_0C^2 2^{\alpha\theta+\beta}s^{-(\alpha\theta+\beta)}\|x\|_{V_\theta}
    \end{align*}
    for all $s\in (0,t]$. Since $\alpha\theta < 1-\beta$ and at least one of the conditions~\ref{assume:bochner-V:itm:separable} and~\ref{assume:bochner-V:itm:closed-ball} holds, Proposition~\ref{prop:bochner-V} now shows that $\int_0^t f_t(s)\dd s$ lies in $V$ with the norm estimate
    \[
        \norm{\int_0^t f_t(s)\dd s}_V \le 2^{\alpha\theta+\beta}M_0C^2  \int_0^t s^{-(\alpha\theta+\beta)}\dd s \norm{x}_{V_\theta}=\frac{2^{\alpha\theta+\beta}M_0 C^2}{1-(\alpha\theta+\beta)} t^{1-(\alpha\theta+\beta)} \|x\|_{V_\theta}.
    \]
    Together with the variation of parameters formula~\eqref{eq:variation-param}, the previous upper bound, the interpolation property \eqref{eq:interpolation-inequality-semigoup}, and the choice of $\theta$ enable us to estimate
    \begin{align*}
        \|S(t)x\|_V &\le Ct^{-\alpha\theta}\|x\|_{V_\theta} + \frac{2^{\alpha\theta+\beta}M_0 C^2}{1-(\alpha\theta+\beta)} t^{1-(\alpha\theta+\beta)} \|x\|_{V_\theta} \\
        &\le C\left(1 + \frac{2M_0 C}{1-(\alpha\theta+\beta)}\right)t^{-\alpha\theta}\|x\|_{V_\theta};
    \end{align*}
    where the second inequality uses $1-\beta>0$ and $t\in (0,1]$.
\end{proof}

\begin{proof}[Proof of Theorem~\ref{thm:perturb-smoothing}]
     Let $\theta\in (0,1)$ satisfy $\alpha\theta < 1-\beta$, and fix an interpolation method $\mathfrak{F}_{1-\theta}$ and interpolation space $V_\theta$ of exponent $1-\theta$ as in Lemma~\ref{lem:interpolation-argument}. Fixing arbitrary $x\in V\cap \dom(A) \subseteq V_\theta \cap \dom(A)$, Lemma~\ref{lem:V-bounded} ensures that
    \[
        m_x \coloneqq \sup_{t\in (0,1]} t^{\alpha}\|S(t)x\|_V
    \]
    is finite. Let us deduce an estimate for $m_x$. Firstly, we observe that
    \[
        \|S(t)x\|_V = \|S(t/2)S(t/2)x\|_V \le 2^{\alpha \theta} C_1 t^{-\alpha\theta}\|S(t/2)x\|_{V_\theta}
    \]
    holds for $t\in (0,1]$ 
    due to~\eqref{eq:S-Vtheta-V} from Lemma~\ref{lem:interpolation-argument}. Since $V_\theta$ is an interpolation space of exponent $1-\theta$, we have the abstract interpolation inequality
    \begin{equation*}
        \|z\|_{V_\theta} \le c_\theta \|z\|^\theta_X \|z\|^{1-\theta}_V \qquad\text{for all }z\in V.
    \end{equation*}
    The preceding two inequalities together imply that
    \begin{align*}
        t^{\alpha}\|S(t)x\|_V &\le 2^{\alpha\theta} C_1 c_\theta \cdot M_1^\theta t^{\alpha(1-\theta)} \|S(t/2)x\|^{1-\theta}_V \|x\|_X^{\theta} \\
        &= C_2 \left(\frac{t}{2}\right)^{\alpha(1-\theta)}\|S(t/2)x\|_V^{1-\theta}\|x\|_X^{\theta};
    \end{align*}
    where $M_1$ is the uniform bound of the perturbed semigroup $(S(t))_{t\ge 0}$ on $X$ for $t\in [0,1]$ and $C_2 \coloneqq 2^{\alpha}C_1 c_\theta M_1^\theta$. Taking supremum over $t\in (0,1]$, it follows that
    \begin{equation*}
        m_x \le C_2 m_x^{1-\theta}\|x\|_X^\theta
    \end{equation*}
    which implies that $m_x \le C_2^{1/\theta}\|x\|_X$ for all $x\in V\cap \dom(A)$. Therefore, setting $\widetilde{C} = C_2^{1/\theta}$, we find for all $x\in V\cap \dom(A)$ and $t\in (0,1]$ that
    \begin{equation*}
        \|S(t)x\|_V \le \widetilde{C} t^{-\alpha}\|x\|_X.
    \end{equation*}
    
    Finally, note that $V\cap \dom(A)$ is dense in $X$. Indeed, for every $x\in \dom(A)$, we have $T(t)x\in V$ for all $t>0$, and $T(t)x\to x$ as $t\downarrow 0$ in the norm of $X$ by strong continuity of the semigroup. Since $\dom(A)$ is dense in $X$, the claim follows. Hence the estimate above extends to all $x\in X$, as desired.
\end{proof}

\begin{remarks} 
    \label{rem:perturb-smoothing}
    (a) Since Lemmata~\ref{lem:V-bounded} and~\ref{lem:interpolation-argument} also hold if $(B, \dom(B))$ is closed, the assumption of immediate differentiability in Theorem~\ref{thm:perturb-smoothing} can therefore be dropped if $(B,\dom(B))$ is closed.

    (b) As Proposition~\ref{prop:perturbed-semigroup} is true for all \emph{Miyadera-Voigt perturbations} (see \cite[Corollary~III.3.15]{EngelNagel}), so Theorem~\ref{thm:perturb-smoothing} also holds if, instead of the assumption~\ref{assume:it:admissibleX}, we more generally assume that $B$ is a Miyadera-Voigt perturbation. Likewise, for immediately differentiable semigroups,  Lemmata~\ref{lem:V-bounded} and~\ref{lem:interpolation-argument} also hold for general Miyadera-Voigt perturbations.
    
    (c) The proof of Theorem~\ref{thm:perturb-smoothing} is an abstract variation of an extrapolation argument that goes back to Coulhon~\cite[Lemme 1]{Coulhon1990}; see also~\cite[p.\ 66]{Arendt2004-survey}. Observe that the interpolation space $V_\theta$ is merely used as a technical tool to obtain desired estimates, and plays no role in the conclusion of the theorem.
\end{remarks}

We can now prove the result stated in Theorem~\ref{thm:ultracontractivity-introduction}. In particular, this shows that the conditions~\ref{assume:bochner-V:itm:separable} or~\ref{assume:bochner-V:itm:closed-ball} in Setting~\ref{setting:admissible} can be dropped from Theorem~\ref{thm:perturb-smoothing} if $B$ is a $V$-invariant bounded operator.

\begin{proof}[Proof of Theorem~\ref{thm:ultracontractivity-introduction}]
    As already explained in Remark~\ref{rem:setting}(a), it is a consequence of the closed graph theorem that $B$ restricts to a bounded operator on $V$, and thus satisfies~\ref{assume:it:admissibleX} and~\ref{assume:it:Vbounds} with $\beta = 0$ and for all $t>0$. Moreover, conditions~\ref{assume:bochner-V:itm:separable} or~\ref{assume:bochner-V:itm:closed-ball} are not required in this case. Indeed, one can establish directly that for all $x\in V$, the integral in the variation of parameters formula~\eqref{eq:variation-param} exists as a Bochner integral in $V$. Consequently, in the proof of Lemma~\ref{lem:interpolation-argument}, one can take $x\in V$ already (instead of the more technical argument requiring $x\in V_\theta\cap\dom(A)$ and Proposition~\ref{prop:bochner-V}). Finally, the immediate differentiability of the unperturbed semigroup $(T(t))_{t\ge 0}$ is not required due to Remark~\ref{rem:perturb-smoothing}(a). Hence the rest of the proof of Theorem~\ref{thm:perturb-smoothing} carries through.
\end{proof}

\begin{remark}[Including a perturbation parameter]
\label{rmk:perturb-with-parameter}
    In perturbation theory, it is common to introduce a parameter $\kappa\in U\subseteq\bbC$, where $U$ is a small neighbourhood of $0$, and study the perturbed operators $A(\kappa) \coloneqq A+\kappa B$. 
    
    Consider operators $A,B$ as in Setting~\ref{setting:admissible} and assume that conditions~\ref{assume:it:ultracontractive}--\ref{assume:it:Vbounds} hold. For $\modulus{\kappa}\le 1$, clearly  $\kappa B$ also satisfies all the relevant conditions in Setting~\ref{setting:admissible}, and moreover one can choose a constant $C\ge 1$ that is independent of $\kappa$. It is then easy to check that Theorem~\ref{thm:perturb-smoothing} holds for the perturbed semigroups $(T_\kappa(t))_{t\ge 0}$ generated by $A(\kappa)$ and crucially, the estimate~\eqref{eq:thm:perturb-smoothing} holds with a constant $\widetilde{C}>0$ independent of $\modulus{\kappa}\le 1$. We make use of these observations in the sequel.
\end{remark}

\section{Analytic perturbations of the spectrum}
    \label{sec:analytic}

This section concerns the analytic dependence of eigenvalues and eigenvectors on the perturbation parameter $\kappa$ for an analytic family of operators $\{A(\kappa)\}_{\kappa}$. While such problems belong to classical perturbation theory, the novelty of our result lies in the stronger conclusions obtained when the operators involved are semigroup generators, and the unperturbed semigroup satisfies additional smoothing conditions as described in Setting~\ref{setting:admissible}. 

The open disk in $\bbC$ with centre $\lambda_0$ and radius $r>0$ is denoted by $B(\lambda_0,r)$.
The standard assumptions in this section are collected in the following:

\begin{setting}
    \label{setting:analytic}
    Let $(T(t))_{t\ge 0}$ be an immediately differentiable $C_0$-semigroup on a Banach space $X$ with generator $A:\dom(A)\subseteq X\to X$ and let $V$ be a Banach space embedding continuously into $X$.
    Additionally, consider a family of operators $\{ B(\kappa) ~\colon \kappa\in B(0,1)\subset\bbC \} \subset \calL(X_1,X)$ with $B(0):=0\restricted{X_1}$.
    We assume the following:
    \begin{itemize}
        \item At least one of the conditions~\ref{assume:bochner-V:itm:separable} and~\ref{assume:bochner-V:itm:closed-ball} on $V$ in Setting~\ref{setting:admissible} is fulfilled.

        \item For each $\kappa$, the conditions~\ref{assume:it:ultracontractive}--\ref{assume:it:Vbounds} in Setting~\ref{setting:admissible} are fulfilled with the operator $B(\kappa)$ instead of $B$.

        \item The mapping $B(0,1)\ni\kappa \mapsto A(\kappa)x$ is holomorphic for each $x\in \dom(A)$; where $A(\kappa)\coloneqq A+B(\kappa)$.
    \end{itemize}
\end{setting}

\begin{remarks}
\label{rem:holomorphic-family}
    (a) The families $\{A(\kappa)\}_{\kappa}$ considered in Setting~\ref{setting:analytic} are known as \emph{holomorphic families of type} (A) in \cite[Section~VII-2]{Kato}. We focus on such perturbations due to their common occurrence in applications -- the condition being relatively easy to verify -- and the fact that we later use perturbations of the form $A(\kappa) = A + \kappa B$, where $A$ and $B$ are as described in Setting~\ref{setting:admissible}. However, as is clear from the proofs below, the results in this section are also valid for holomorphic families of perturbations as defined in \cite[p.\ 366]{Kato}.

    (b) It is also clear from the proof below that Theorem~\ref{thm:eigenfunction-lower-bound}(b) holds for the perturbed operator $A+B$ whenever $\hat{\delta}(A+B,A)$, the \emph{gap} between $A+B$ and $A$, is sufficiently small (see Appendix~\ref{appendix:spectrum}). 
\end{remarks}

\begin{theorem}
    \label{thm:analyticity-spectrum}
    In Setting~\ref{setting:analytic},
    suppose that $\lambda_0\in\spec(A)$ is a pole of the resolvent $\Res(\argument,A)$ and an algebraically simple eigenvalue, let $r>0$ be such that $\overline{B(\lambda_0,r)}\cap\spec(A) = \{\lambda_0\}$, and for each $\kappa \in B(0,1)$ set
    \[
        \beta_{\kappa} \coloneqq \sup_{\modulus{\lambda-\lambda_0}=r} \norm{B(\kappa) \Res(\lambda,A)}.
    \]

    If $\beta_{\kappa}\to 0$ as $\modulus{\kappa}\to 0$, then there exists $\delta>0$ such that for every $\kappa\in B(0,\delta)$, there is a unique spectral value $\lambda(\kappa)$ of the perturbed operator $A(\kappa)$ in $B(\lambda_0,r)$ that is a pole of the resolvent $\Res\big(\argument,A(\kappa)\big)$ and an algebraically simple eigenvalue. The corresponding spectral projection $P(\kappa)$ satisfies $\rg P(\kappa)\subseteq V$ and the maps
    \begin{align}
        B(0,\delta)\ni\kappa &\mapsto \lambda(\kappa)\in\bbC \\
        B(0,\delta)\ni\kappa &\mapsto P(\kappa)\in \calL(V) \label{eq:analytic-perturb-eigenvector}
    \end{align}
    are analytic.
\end{theorem}

Note that the mapping $\kappa \mapsto \beta_{\kappa}$ need not be continuous at $\kappa =0$. Therefore, the fact that $B(0)=0\restricted{X_1}$ does not make the condition $\beta_{\kappa}\to 0$ as $\modulus{\kappa}\to 0$ redundant.

Due to the analyticity assumption on the family $\{A(\kappa)\}$ in Setting~\ref{setting:analytic}, the assertions of Theorem~\ref{thm:analyticity-spectrum} -- except with $\calL(X)$ instead of $\calL(V)$ in~\eqref{eq:analytic-perturb-eigenvector} --  follow from a result of Kato \cite[Theorem~VII-1.7]{Kato}. Therefore, the main novelty of Theorem~\ref{thm:analyticity-spectrum} is extending the analyticity of the spectral projection map with co-domain $\calL(V)$.

The proof of Theorem~\ref{thm:analyticity-spectrum} requires the following basic topological fact. Recall that for a Banach space $Y$, a subset $W\subseteq Y'$ is called \emph{separating} if for every $0\ne y\in Y$ there exists $w\in W$ such that $\braket{w,y}_{Y',Y} \ne 0$. Equivalent conditions for a subspace of a dual to be separating are available in \cite[Theorem~4.7(b)]{Rudin-FA}.

\begin{lemma}
\label{lem:weakstar-dense}
    Let $X, Y$ be Banach spaces such that $Y$ embeds continuously and densely into $X$. Then $X'$ is a separating subspace of $Y'$.
\end{lemma}

\begin{proof}
    Let $\iota:Y\hookrightarrow X$ denote the natural injection. Since $\iota$ has dense range by assumption, the dual operator $\iota'$ is injective, and thus it makes sense to consider $X'$ as a subspace of $Y'$. Strictly speaking, we need to prove that $\iota'(X')$ is separating in $Y'$. If this is not the case, then there exists $0\ne y\in Y$ such that 
    \[ 0=\braket{\iota'(x'),y}_{Y',Y} = \braket{x',\iota(y)}_{X',X}  \]
    for all $x'\in X'$. Hence $\iota(y)$ is an element of the \emph{pre-annihilator}
    \[
        {}^\perp X' \coloneqq \{x\in X ~\colon \braket{x',x}=0\text{ for all }x'\in X'\}=\{0\}.
    \]
    Thus $\iota(y) = 0$ and in turn, $y=0$ since $\iota$ is injective. This is a contradiction.
\end{proof}

\begin{proof}[Proof of Theorem~\ref{thm:analyticity-spectrum}]
    The assumption of holomorphy in Setting~\ref{setting:analytic} ensures that the family $\{A(\kappa)\}_{\kappa}$ is also holomorphic in the sense of \cite[Section~VII-1.2]{Kato}; cf. \cite[p.\ 375]{Kato}. On the other hand, the admissibility assumption~\ref{assume:it:admissibleX} in Setting~\ref{setting:admissible} guarantees $B$ is $A$-bounded (Remark~\ref{rem:setting}(b)).
    These observations allow us to combine a spectral result for analytic perturbations due to Kato~\cite[Theorem~VII-1.7]{Kato} with Proposition~\ref{prop:spectral-perturb} to deduce that there exist $\delta>0$ and analytic maps $B(0,\delta)\ni\kappa \mapsto \lambda(\kappa)\in\bbC$ and $B(0,\delta)\ni\kappa \mapsto P(\kappa)\in \calL(X)$ satisfying the desired conclusions of the theorem, except with $\calL(X)$ instead of $\calL(V)$ in~\eqref{eq:analytic-perturb-eigenvector}. Thus the remainder of the proof consists of extending the analyticity of the map $P(\argument)$ with co-domain $\calL(V)$. For this purpose, we verify the conditions of the Arendt-Nikolskii theorem~\cite[Theorem 3.1]{ArendtNikolskii00}, i.e., we show that the function in~\eqref{eq:analytic-perturb-eigenvector} is locally bounded and $\braket{\varphi, P(\argument)v}_{V',V}$ is analytic for each $v\in V$ and for all $\varphi$ in a separating subset of $V'$.
    
    \emph{Step 1}: In this step, we prove for every compact subset $U\subseteq B(0,\delta)$ that
    \begin{equation}
    \label{eq:locally-bounded}
        \sup_{\kappa\in U} \|P(\kappa)\|_{\calL(V)} < \infty.
    \end{equation}
    Let $v\in X$ be arbitrary. Then $P(\kappa)v$ is either $0$ or is an eigenvector for the eigenvalue $\lambda(\kappa)$. It is moreover an element of $V$, since Theorem~\ref{thm:perturb-smoothing} implies that the perturbed semigroup $(T_\kappa(t))_{t\ge 0}$ generated by $A(\kappa)$ maps $X$ into $V$ and hence
    \begin{equation*}
        P(\kappa)v = e^{-\lambda(\kappa)}T_\kappa(1)P(\kappa)v \in V.
    \end{equation*}
    This shows that $P(\kappa)X\subseteq V$. Now the estimate~\eqref{eq:thm:perturb-smoothing} for the operator $T_\kappa(1)$ yields
    \begin{equation}
    \label{eq:v-kappa-estimate}
        \|P(\kappa)v\|_V \le  \big\vert e^{-\lambda(\kappa)}\big\vert \|T_\kappa(1)\|_{X\to V}\|P(\kappa)v\|_X \le \widetilde{C} e^{-\re\lambda(\kappa)}\|P(\kappa)\|_{\calL(X)} \|v\|_X
    \end{equation}
    for every $\kappa\in B(0,\delta)$, where the constant $\widetilde{C}>0$ can be chosen to be independent of $\kappa$ due to Remark~\ref{rmk:perturb-with-parameter}. Inequality~\eqref{eq:v-kappa-estimate} clearly shows that $P(\kappa)\in\calL(X,V)$ for every $\kappa\in B(0,\delta)$. Also, analyticity of the maps $\kappa\mapsto\lambda(\kappa)\in\bbC$ and $\kappa\mapsto P(\kappa)\in \calL(X)$ implies that they are continuous and hence uniformly bounded on any compact subset $U\subseteq B(0,\delta)$. Choose $\Lambda, M\in [0,\infty)$, depending only on $U$, such that
    \begin{equation*}
        \sup_{\kappa\in U}|\lambda(\kappa)| \le \Lambda, \quad \sup_{\kappa\in U}\|P(\kappa)\|_{\calL(X)} \le M.
    \end{equation*}
    With the help of~\eqref{eq:v-kappa-estimate}, we may now infer that 
    \[
        \sup_{\kappa\in U}\|P(\kappa)\|_{X\to V} \le \widetilde{C}M e^\Lambda < \infty.
    \]

    \emph{Step 2}: Let $v\in V$ be arbitrary. Since $P(\kappa)v\in V$ for every $\kappa\in B(0,\delta)$, it holds that $\braket{\varphi, P(\kappa)v}_{X',X} = \braket{\varphi, P(\kappa)v}_{V',V}$ for every $\varphi\in X'$. Since $\kappa\mapsto P(\kappa)v$ is analytic, and hence weakly analytic, with values in $X$, the scalar-valued map $\kappa\mapsto \braket{\varphi, P(\kappa)v}_{V',V}$ is analytic on $B(0,\delta)$ for every $\varphi\in X'$. Observe that $V$ is densely embedded in $X$ thanks to condition~\ref{assume:it:ultracontractive} in Setting~\ref{setting:admissible} in combination with the strong continuity of $(T(t))_{t\ge 0}$ on $X$. It follows from Lemma~\ref{lem:weakstar-dense} that $X'$ is a separating subset of $V'$. The Arendt-Nikolskii theorem~\cite[Theorem 3.1]{ArendtNikolskii00} now shows that $\kappa\mapsto P(\kappa)v\in V$ is analytic for every $v\in V$.

    Lastly, the pointwise analyticity is equivalent to the analyticity of the map $\kappa\to P(\kappa)\in\calL(V)$ by~\cite[Proposition A.3]{ABHN}.
\end{proof}

\begin{remark}
    The general proof strategy for obtaining analyticity of the spectral projections $P(\argument)$ is loosely adapted from a specific problem considered in~\cite[Proposition 4.11]{DanersGlueckMui2023}, where $X=L^2(\Omega)$ and $V=\Cont(\overline{\Omega})$ for a bounded domain $\Omega\subset\bbR^n$ with smooth boundary. There, the authors worked directly with eigenvectors (instead of the more abstract object of spectral projections) and the local boundedness required to apply the Arendt-Nikolskii theorem was obtained without using semigroups, instead relying on elliptic regularity theory and a Sobolev embedding theorem.
\end{remark}

When the Banach space from Setting~\ref{setting:analytic} is a function space such as $L^p$ or $\Cont(K)$, or more generally, a Banach lattice, then Theorem~\ref{thm:analyticity-spectrum} can be used to obtain lower bounds on positive eigenfunctions of perturbed operators. 

\begin{theorem}
\label{thm:eigenfunction-lower-bound}
    In Setting~\ref{setting:analytic}, assume that $X=E$ is a complex Banach lattice, the operator $A$ and $B(\kappa)$ are real for each $\kappa \in \bbR$ with $\modulus{\kappa}<1$, and let $V=E_u$ for some $u\in E_+$.

    Suppose that $\lambda_0\in\spec(A) \cap \bbR$ is a pole of the resolvent $\Res(\argument,A)$ and an algebraically simple eigenvalue with an eigenvector $v_0$ satisfying $v_0 \ge cu$ for some constant $c>0$.
    Fix $r>0$ such that $\overline{B(\lambda_0,r)}\cap\spec(A) = \{\lambda_0\}$ and
    set
    \[
        \beta_{\kappa} \coloneqq \sup_{\modulus{\lambda-\lambda_0}=r} \norm{B(\kappa) \Res(\lambda,A)}
    \]
    for each $\kappa \in \bbR$ with $\modulus{\kappa}<1$.
    
    If $\beta_{\kappa}\to 0$ as $\kappa\to 0$,
    then
    there exists $\delta>0$ such that for all $\kappa \in (-\delta,\delta)$, the following assertions hold:
    \begin{enumerate}[\upshape (a)]
        \item The perturbed operator $A(\kappa)$ has a unique algebraically simple eigenvalue $\lambda(\kappa)$ in the open interval $(\lambda_0-r,\lambda_0+r)$ with an associated eigenvector $v(\kappa)\in E_u$.
        Moreover, $\overline{B(\lambda_0,r)} \cap \spec( A(\kappa)) = \{\lambda(\kappa)\}$.
        
        \item One may choose an eigenvector $v(\kappa)$ corresponding to $\lambda(\kappa)$ such that
        \begin{equation*}
            v(\kappa) \ge \frac{c}{2}u.
        \end{equation*}
    \end{enumerate}
\end{theorem}

Example~\ref{exas:closed-unit-ball}\ref{exas:closed-unit-ball:itm:principal-ideal} below shows that the closed unit ball of the principal ideal $E_u$ is indeed closed in the Banach lattice $E$, and so $E_u$ satisfies the condition~\ref{assume:bochner-V:itm:closed-ball} of Setting~\ref{setting:admissible}. This enables us to choose $V=E_u$ in Setting~\ref{setting:analytic}.

\begin{proof}[Proof of Theorem~\ref{thm:eigenfunction-lower-bound}]
    Firstly, by Theorem~\ref{thm:analyticity-spectrum}, there exists $\varepsilon_0>0$ such that the first assertion is true for all $\kappa\in (-\varepsilon_0,\varepsilon_0)$. Indeed, $\lambda(\kappa)$ is real, since $\lambda_0\in \bbR$ and the operators $A$ and $B(\kappa)$ are real for all $\kappa\in (-\varepsilon_0,\varepsilon_0)$; see Proposition~\ref{prop:spectral-perturb}.

    Moreover, Theorem~\ref{thm:analyticity-spectrum} also implies that the map
    \begin{equation*}
        (-\varepsilon_0,\varepsilon_0) \ni \kappa \mapsto v(\kappa)\coloneqq P(\kappa)v_0 \in E_{u}
    \end{equation*}
    is analytic, and in particular continuous. Note that $P(0)v_0 = v_0$. Hence, we may choose $\delta < \varepsilon_0$ sufficiently small so that
    \begin{equation*}
        \|v(\kappa)-v_0\|_{E_u} < \frac{c}{2}
    \end{equation*}
    for all $\kappa\in (-\delta,\delta)$. Recalling the definition of the gauge norm on $E_u$, this implies that $|v(\kappa)-v_0| < \frac{c}{2}u$, and therefore
    $
        v(\kappa) \ge \frac{c}{2}u
    $
    for all $\kappa\in (-\delta,\delta)$.
\end{proof}

We can now prove the result stated in Theorem~\ref{thm:analyticity-introduction}. In particular, this shows that the immediate differentiability of the semigroup and the conditions~\ref{assume:bochner-V:itm:separable} or~\ref{assume:bochner-V:itm:closed-ball} in Setting~\ref{setting:analytic} can be dropped from Theorem~\ref{thm:eigenfunction-lower-bound} whenever $B$ is a $E_u$-invariant bounded operator.

\begin{proof}[Proof of Theorem~\ref{thm:analyticity-introduction}]
    We may repeat the proof of Theorem~\ref{thm:analyticity-spectrum}, but employing Theorem~\ref{thm:ultracontractivity-introduction} instead of Theorem~\ref{thm:perturb-smoothing}. Then the assertions of Theorem~\ref{thm:analyticity-spectrum} hold in the present situation as well. The result can therefore be argued exactly as in the proof of Theorem~\ref{thm:eigenfunction-lower-bound}.
\end{proof}

\section{Applications and examples: elliptic operators}
    \label{sec:applications}
Semigroups generated by elliptic operators constitute a natural class of ultracontractive semigroups that satisfy the conditions outlined in Setting~\ref{setting:admissible}. Although Davies previously examined the robustness of ultracontractivity for second-order elliptic operators \cite{Davies1986-perturbation}, our results accommodate some non-standard perturbations of second-order operators as well as higher-order elliptic operators.

\subsection{Second-order elliptic operators}

Let $\Omega\subseteq\bbR^n$ be a bounded domain with Lipschitz boundary, and consider a standard uniformly elliptic operator in divergence form, given formally by
\begin{equation}
\label{eq:unif-elliptic}
    Lu \coloneqq -\divergence(A\nabla u) = -\sum_{i,j=1}^n \frac{\partial}{\partial x_i}\left( a_{ij}(x)\frac{\partial }{\partial x_j}u \right),
\end{equation}
where the coefficient matrix $A = (a_{ij}) \in L^\infty(\Omega;\bbR^{n\times n})$ satisfies
\begin{equation*}
    \re \sum_{i,j=1}^n a_{ij}(x)\xi_i \overline{\xi_j} \ge \alpha|\xi|^2 \qquad\text{for all }x\in\Omega,\xi\in\bbC^n
\end{equation*}
for some constant $\alpha>0$. The standard lower-order terms are omitted for the sake of simplicity, as they do not cause additional difficulties in our framework. 
We equip $L$ with Neumann-Robin boundary conditions given (formally) by
\begin{equation*}
    \nu\cdot A\nabla u + \beta u = 0 \quad\text{on }\partial\Omega,
\end{equation*}
where $\beta\in L^\infty(\partial\Omega;\bbR)$ and $\nu$ denotes the outer unit normal on $\partial\Omega$. 

The operator can be rigorously defined by the sesquilinear form
\begin{equation}
\label{eq:form-robin}
    \left\{ ~ \begin{aligned}
    \dom(\fra) &\coloneqq H^1(\Omega) \subset L^2(\Omega) \\
    \fra[u,v] &\coloneqq \int_\Omega A\nabla u\cdot\overline{\nabla v} \dd x + \int_{\partial\Omega} \beta u \overline{v} \dd\sigma \quad \text{for all }u,v\in H^1(\Omega);
    \end{aligned}\right.
\end{equation}
where $\sigma$ denotes the surface measure on $\partial\Omega$.
Clearly, $\fra$ is a densely defined closed form on $L^2(\Omega)$, and it is straightforward to see, with the help of the trace inequality (valid on Lipschitz domains, see e.g.~\cite[Theorem 18.1]{Leoni}), that there exist constants $C,\gamma>0$ such that
\begin{align}
    \big|\fra[u,v]\big| \le C\|u\|_{H^1(\Omega)} \|v\|_{H^1(\Omega)}& \quad\text{and} \\
    \re\fra[u,u] + \gamma\|u\|^2_{L^2(\Omega)} \ge \alpha\|u\|^2_{H^1(\Omega)}&
\end{align}
for all $u,v\in H^1(\Omega)$. These inequalities show that the form $\fra$ is continuous and $\fra + \gamma$ is coercive on $H^1(\Omega)$. We refer to the monograph of Ouhabaz~\cite[Chapter 1]{Ouhabaz} for general terminology in the theory of forms.
The operator on $L^2(\Omega)$ associated to the form $\fra$ is now defined by
\begin{equation}
\label{eq:L-via-form}
    \left\{\, \begin{aligned}
        \dom(L) &\coloneqq \{u\in H^1(\Omega) ~\colon\exists v\in L^2(\Omega), \fra[u,w] = \braket{v,w}\,\forall w\in H^1(\Omega)\} \\
        Lu &= v.
    \end{aligned}\right.
\end{equation}
For sufficiently smooth coefficients $A$ and functions $u$, one sees via integration by parts that $Lu$ is indeed given by the formal differential operator introduced in~\eqref{eq:unif-elliptic}, and thus there is no harm in reusing the notation $L$ for the operator defined in~\eqref{eq:L-via-form}.

In the above setting, the well-known Lumer-Phillips theorem shows that $-L$ is the generator of a $C_0$-semigroup $(T(t))_{t\ge 0}$ on $L^2(\Omega)$, see e.g.~\cite[Proposition~1.22 and Theorem~1.49]{Ouhabaz}. We collect further important properties of the semigroup below. Note that we use the term \emph{positive} for a linear operator to mean \emph{positivity-preserving}.
\begin{proposition}
\label{prop:heat-semigroups}
    On $L^2(\Omega)$, the semigroup $(T(t))_{t\ge 0}$ generated by $-L$ is positive and analytic. Moreover, the following assertions hold.
    \begin{enumerate}[\upshape (a)]
        \item It holds that $T(t)L^2(\Omega) \subset \Cont(\overline{\Omega})$ for all $t>0$, and there exists a constant $C>0$ such that
        \begin{equation*}
            \|T(t)\|_{2\to\infty} \le Ct^{-n/4}
        \end{equation*}
        for all $t\in (0,1]$.
        In addition, $(T(t))_{t\ge 0}$ restricts to a holomorphic $C_0$-semigroup on $\Cont(\overline{\Omega})$.
        
        \item The semigroup extrapolates to a consistent family of semigroups $(T_p(t))_{t\ge 0}$ on $L^p(\Omega)$ for all $p\in [1,\infty]$ with $T_2 = T$. Moreover, the strong continuity holds for all $p\ne\infty$ and the analyticity holds for $1<p<\infty$.
    \end{enumerate}
\end{proposition}

\begin{proof}
    Positivity follows from the well-known Beurling-Deny criterion~\cite[Theorem 2.6]{Ouhabaz}. Since the form $\fra + \gamma$ is coercive, it is in particular \emph{sectorial}, and therefore the analyticity of the semigroup follows from~\cite[Theorem~1.52]{Ouhabaz}.
    Moreover,~(a) follows from the work of Nittka~\cite[Theorems~4.1 and~4.3]{Nittka}, and the ultracontractivity estimate is classical; see~\cite[Theorem 7.3.2]{Arendt2004-survey} for instance.

    (b) The Gaussian estimates obtained in~\cite{ArendtterElst1997} yield the strong continuity of the semigroup on $L^p(\Omega)$, $1\le p<\infty$, and also the extrapolation property; see, in particular, \cite[Theorem~4.9]{ArendtterElst1997}. Additionally, we note that the assumption of positivity of $\beta$ in that paper can be removed, since on a bounded Lipschitz domain, one can transform~\eqref{eq:form-robin} into an equivalent problem with a new Robin coefficient $\tilde{\beta}$ that is positive. Details can be found in~\cite[Theorem~2.2 and Lemma~3.2]{Daners2009}.

    Finally, the extrapolation of holomorphy for $1<p<\infty$ is a classical argument using the Stein interpolation theorem, see e.g.\ the proof of~\cite[Theorem 1.4.2]{Davies1990}. An overview of many key ideas in the study of extrapolation of semigroups can be found in \cite[Sections~7.2--7.4]{Arendt2004-survey}.
\end{proof}

Inspired by an example considered by Rellich in his monograph~\cite[p.\ 104]{Rellich1969} in the context of perturbation of symmetric forms, we now illustrate how Theorem~\ref{thm:perturb-smoothing} can be applied to some non-standard perturbations of second-order elliptic operators.

\begin{example}[Perturbation by unbounded rank-one operators]
    Consider the second-order elliptic operator $L=-\frac{\dd^2}{\dd x^2}$ with Neumann boundary conditions on the interval $(-\pi,\pi)$. In other words, $L$ is the operator arising from the sesquilinear form
    \[
        \fra[u,v] \coloneqq \int_{-\pi}^\pi u'\overline{v'}\dd x, \qquad u,v\in H^1(-\pi,\pi).
    \]
    Let $V \coloneqq \Cont([-\pi,\pi])$, and denote the Dirac delta functional at $0$ by $\delta_0$. We now define $B: V \to X$ by
    \begin{equation*}
        B \coloneqq \braket{-\delta_0,\argument}_{V',V}  \one_{(-\pi,\pi)}.
    \end{equation*}
    We claim that the operator $\calL \coloneqq -L + B$ generates a semigroup $(S(t))_{t\ge 0}$ on $X\coloneqq L^2(-\pi,\pi)$ that maps $X$ into $V$ and fulfils
    \[
        \|S(t)\|_{X\to V}\le \widetilde Ct^{-1/4}
    \]
    for some constant $\tilde C>0$ and all $t\in (0,1]$.
    Indeed, note firstly that
    \[
        \dom(-L) \subset H^1(-\pi,\pi) \hookrightarrow V
    \]
    (cf. \cite[Theorem 8.2]{Brezis}), and that the restriction of $B$ is continuous on $\dom(-L)$. Clearly $V$ is separable and embeds continuously into $X$. Finally, in light of Proposition~\ref{prop:heat-semigroups}, it is straightforward to check that the assumptions of Theorem~\ref{thm:perturb-smoothing} are satisfied, and hence $\calL$ satisfies the claim.
    
    Formally speaking, $\calL$ can be realised as the differential ``operator''
    \begin{equation*}
        \calL u = u'' - \delta_0 u,
    \end{equation*}
    which can be viewed as a Schr\"{o}dinger operator with the Dirac delta ``function'' $\delta_0$ as a potential. We point out in addition that the perturbation $B$ is not closable in $X$, i.e.\ the closure in $X\times X$ of the graph of $B$ does not define a single-valued operator.
\end{example}

A more substantial example will arise in Theorem~\ref{thm:frac-powers} below. We defer its presentation to the next subsection, since it turns out that we can treat second- and higher-order operators in the same general framework.

\subsection{Higher-order elliptic operators}

To avoid technical difficulties in this subsection, we restrict attention to a bounded domain $\Omega\subset\bbR^n$ with smooth boundary. As a prototype of a higher-order elliptic operator, we now consider the polyharmonic operator $(-\Delta)^m$  for an integer $m\ge 2$ with Dirichlet boundary conditions. This operator is formally given by
\begin{equation*}
    \left\{ \begin{aligned}
        &Lu \coloneqq (-\Delta)^m u  \quad &&\text{in }\Omega \\
        &D^\alpha u = 0  \quad &&\text{on }\partial\Omega,
    \end{aligned}\right.
\end{equation*}
with $|\alpha| \le m-1$ (we use the standard multi-index notation). Again, for simplicity, we omit lower-order terms which present no additional difficulty in our setting.

Similar to the second-order case, $L$ arises from a densely defined, closed sesquilinear form given by
\begin{equation}
\label{eq:polyharmonic-form}
    \left\{ \begin{aligned}
        \dom(\fra) &\coloneqq H^m_0(\Omega) \\
        \fra[u,v] &\coloneqq \begin{cases}
            \int_\Omega \Delta^k u \, \overline{\Delta^k v} \dd x \quad &\text{if } m=2k; \\[0.5em]
            \int_\Omega \nabla(\Delta^k u) \cdot \overline{\nabla(\Delta^k v)} \dd x \quad &\text{if }m=2k+1.
        \end{cases}
    \end{aligned}\right.
\end{equation}
This form is symmetric so that the associated operator
\begin{equation}
\label{eq:L-via-form-poly}
    \left\{\, \begin{aligned}
        \dom(L) &\coloneqq \{u\in H^m_0(\Omega) ~\colon\exists v\in L^2(\Omega), \fra[u,w] = \braket{v,w}\,\forall w\in H^m_0(\Omega)\} \\
        Lu &= v.
    \end{aligned}\right.
\end{equation}
is self-adjoint on $L^2(\Omega)$. Using Gagliardo-Nirenberg inequalities for intermediate derivatives, cf.\ \cite[Theorem 5.2]{AdamsFournier2003} and multiple applications of the classical Poincar\'{e} inequality $\|u\|_2 \le c\|\nabla u\|_2$ for $u\in H^1_0(\Omega)$, one can show that the norm
\begin{equation*}
    u\mapsto \begin{cases}
        \|\Delta^k u\|_2 \quad &\text{if }m=2k; \\
        \|\nabla(\Delta^k u)\|_2 \quad &\text{if }m=2k+1
    \end{cases}
\end{equation*}
is equivalent to the usual `full' $H^m$ norm on $H^m_0(\Omega)$. Consequently the form~\eqref{eq:polyharmonic-form} is coercive on $H^m_0(\Omega)$.

Let $L$ be an elliptic operator as defined either in~\eqref{eq:L-via-form} or~\eqref{eq:L-via-form-poly}. By replacing the associated form $\fra$ with $\fra+\gamma$ for a suitable $\gamma\in\bbR$, we may assume without loss of generality that $\fra$ is coercive on the form domain, namely there exists $\alpha>0$ such that
\begin{equation*}
    \re\fra[u,u] \ge \alpha\|u\|^2_{\dom(\fra)} \qquad\text{for all }u\in \dom(\fra);
\end{equation*}
where $\dom(\fra)=H^1(\Omega)$ in the second-order case and $\dom(\fra)=H^m_0(\Omega)$ in the polyharmonic case. Then the operator $L$ is injective and sectorial, and hence the fractional powers $L^s$, $s\in (0,1)$, are well-defined by the Dunford-Riesz functional calculus. Moreover, these operators have bounded inverses, and are also densely defined and closed. Details can be found in e.g.~\cite[Section II.5c]{EngelNagel} or~\cite[Section 2.6]{Pazy}.

\begin{theorem}
    \label{thm:frac-powers}
    Let $L$ be an elliptic operator as defined either in~\eqref{eq:L-via-form} or~\eqref{eq:L-via-form-poly} such that the associated form $\fra$ is coercive.

    For $s\in (0,1)$, the \emph{mixed local-nonlocal} operator
    \begin{equation*}
        -L + L^s
    \end{equation*}
    generates a $C_0$-semigroup $(S(t))_{t\ge 0}$ on $L^2(\Omega)$ such that
    \begin{equation*}
        \|S(t)\|_{2\to\infty} \le \widetilde Ct^{-n/{4m}}, \quad t\in (0,1] 
    \end{equation*}
    for a suitable constant $\widetilde C=C(s)>0$; note that $m=1$ in the second-order case.
\end{theorem}
    
For the proof of Theorem~\ref{thm:frac-powers}, we collect some properties of $-L$ in the following proposition.
\begin{proposition}
\label{prop:polyharmonic-semigroup-ultrac}
    Let $L$ be the polyharmonic operator~\eqref{eq:L-via-form-poly} and let $(T(t))_{t\ge 0}$ be the semigroup on $L^2(\Omega)$ generated by $-L$. Then $(T(t))_{t\ge 0}$ extrapolates to an analytic $C_0$-semigroup on $L^p(\Omega)$ for all $1\le p <\infty$ whose generator $L_p$ has domain
    \begin{equation*}
        \dom(L_p) = W^{2m,p}(\Omega)\cap W^{m,p}_0(\Omega).
    \end{equation*}
    Furthermore, $T(t)L^2(\Omega)\subset\Cont_0(\Omega)$ for all $t>0$ and there exists $C>0$ such that
    \begin{equation*}
        \|T(t)\|_{2\to\infty} \le C t^{-n/4m}
    \end{equation*}
    for all $t\in (0,1]$. Finally, $T$ also extrapolates to an analytic $C_0$-semigroup on $\Cont_0(\Omega)$.
\end{proposition}

\begin{proof}
    Elliptic regularity for the polyharmonic Dirichlet boundary problem
    \begin{equation*}
        \left\{ \begin{aligned}
            (-\Delta)^m u &= f \quad\text{in }\Omega \\
            D^\alpha u &= 0 \quad\text{on }\partial\Omega
        \end{aligned}\right.
    \end{equation*}
    and much more general higher-order elliptic problems on smooth domains is a classical topic, of which a systematic study was presented in the seminal paper~\cite{ADN-1} of Agmon, Douglis, and Nirenberg. A self-contained account of this theory is given in~\cite[Chapters~4 and~5]{Tanabe}; see also~\cite[Chapter~2]{GGS} for a concise summary of basic results. From the $L^p$ elliptic estimates of~\cite[Theorem~5.3]{Tanabe}, and the fact that weak solutions are \emph{a priori} in $H^m_0(\Omega)$, we obtain
    \begin{equation*}
        \dom(L_p) = W^{2m,p}(\Omega)\cap W^{m,p}_0(\Omega);
    \end{equation*}
    note that $m$ in~\cite{Tanabe} corresponds to our $2m$. The same results lead to the extrapolated analytic semigroups on $L^p(\Omega)$ for $1<p<\infty$ -- see~\cite[Theorem~5.6]{Tanabe} and the preceding proofs -- as well as the estimate
    \begin{equation*}
        |K(t,x,y)| \le \frac{C}{t^{n/2m}}\exp\left(-c\frac{|x-y|^{2m/(2m-1)}}{t^{1/(2m-1)}}\right) \quad t>0, (x,y)\in\Omega\times\Omega
    \end{equation*}
    for the integral kernel associated to the semigroup $T$; cf.~\cite[Equation~(5.145)]{Tanabe}.
    
    Analyticity of the semigroup yields that for every $t>0$ and $k\in\bbN$, we have
    \begin{equation*}
        T(t)L^2(\Omega) \subset \dom(L^k) \subset H^{2mk}(\Omega)\cap H^m_0(\Omega).
    \end{equation*}
    The latter space embeds continuously into $\Cont(\overline{\Omega})$ for sufficiently large $k$, so that the boundary condition $u=0$ on $\partial\Omega$ is satisfied in the classical sense. Hence $T(t)L^2(\Omega)\subset \Cont_0(\Omega)$ as claimed. This yields the ultracontractivity, and the quantitative estimate is implied by the heat kernel bound.
    
    Finally, the kernel estimates lead to the strong continuity and analyticity of the semigroups on $L^1(\Omega)$ and $\Cont_0(\Omega)$. These properties are shown in~\cite[Sections~5.4 and~5.5]{Tanabe}; see also~\cite[Theorem~1.3]{AE19} and its proof~\cite[p.\ 779--780]{AE19}, which, having been implemented in an abstract way, can be adapted to our situation.
\end{proof}

\begin{remarks}
\label{rmk:polyharmonic}
    (a) Clearly, Proposition~\ref{prop:polyharmonic-semigroup-ultrac} also works with $m=1$ and thus contains the Dirichlet Laplacian as a special case.
    
    (b) Unlike the second-order case, semigroups generated by higher-order elliptic operators are, as a rule, never positive. This is due to the simple reason that the Beurling-Deny criterion is violated, see e.g.\ \cite[Proposition 3.5]{DenkKunzePloss2021}. However, as mentioned in the introduction of the article, such semigroups can nevertheless be \emph{eventually positive}.
\end{remarks}

\begin{proof}[Proof of Theorem~\ref{thm:frac-powers}]
    It suffices to verify the conditions of Theorem~\ref{thm:perturb-smoothing} with $X=L^2(\Omega)$ and $V=\Cont(\overline{\Omega})$ in the second-order case, and $V=\Cont_0(\Omega)$ for the Dirichlet polyharmonic operator. 

    (a) Ultracontractivity and analyticity of the semigroups were indicated in Propositions~\ref{prop:heat-semigroups} and~\ref{prop:polyharmonic-semigroup-ultrac} respectively.

    (b) By a standard result in the theory of analytic semigroups (see for instance~\cite[Chapter 2, Theorem 6.13]{Pazy}), for every $s\in (0,1)$ there exists $M_s>0$ such that
    \begin{equation}
    \label{eq:fractional-power-estimate}
        \|L^s T(t)\|_{2\to 2} \le M_s t^{-s} \qquad t\in (0,1].
    \end{equation}
    Since the function $t\mapsto t^{-s}$ is integrable near $0$ when $s\in (0,1)$, the condition~\ref{assume:it:admissibleX} in Setting~\ref{setting:admissible} is fulfilled by integrating the above inequality.

    (c) As mentioned in Propositions~\ref{prop:heat-semigroups} and~\ref{prop:polyharmonic-semigroup-ultrac}, the semigroup $T$ extrapolates to an analytic $C_0$-semigroup on $V$. Hence we may use the estimate~\eqref{eq:fractional-power-estimate} with the norm $\|\argument\|_{\infty\to\infty}$ instead of $\norm{\argument}_{2\to 2}$ on the left-hand side.

    (d) In both cases, $V$ is separable, i.e., condition~\ref{assume:bochner-V:itm:separable} is satisfied.
\end{proof}

\begin{remark}
    The arguments given above are of an abstract nature and thus could be used in diverse situations. One problem is that the fractional powers $L^s$ for a general elliptic operator $L$ are often difficult to handle in concrete problems. However, in our setting, since $L$ has compact resolvent, the fractional powers can be realised in a concrete way via the spectral decomposition of $L$. If $\{\lambda_k\}_{k\in\bbN} \subset [0,\infty)$ denote the eigenvalues of $L$ in increasing order and $\{\varphi_k\}_{k\in\bbN}$ is an associated orthonormal system of eigenfunctions in $L^2(\Omega)$, then we have
    \begin{equation*}
        L^s \coloneqq \sum_{k=1}^\infty \lambda_k^s \braket{\varphi_k, \argument}\varphi_k.
    \end{equation*}
    In the special case $L=-\Delta$, the operator $L^s$ is sometimes known as the \emph{spectral fractional Laplacian} and is of independent interest in PDE analysis. We point out in particular~\cite{CaffarelliStinga2016} and the references therein.

    By now, the literature on the fractional Laplacian is enormous and is a sub-field in its own right. In recent decades, mixed local-nonlocal operators have also attracted a lot of attention; see, for instance, the recent systematic study~\cite{Biagi-etal2022}. However, one must be careful of the many different approaches to defining the fractional Laplacian, so it is not automatically guaranteed that the abstract arguments above can be applied to other problems encountered in this area.
\end{remark}

\section{Applications and examples: eventually positive semigroups}
    \label{sec:eventual-positivity}

Evolution equations associated with second-order elliptic operators typically have a positivity-preserving property; in short, the time evolution is governed by a semigroup of positive operators. However, certain semigroups linked to higher-order operators, such as the Dirichlet bi-Laplacian, exhibit the more nuanced property of \emph{eventual positivity}. This means that orbits starting from positive initial data may change sign initially, but eventually become and remain positive. This phenomenon, first observed in infinite dimensions for the Dirichlet-to-Neumann semigroup \cite{Daners2014}, has gained significant attention over the past decade; see \cite{Glueck2022} for a survey.

In general, while the positivity of a semigroup is preserved under positive perturbations, eventual positivity does not necessarily exhibit the same robustness \cite[Example~2.1]{DanersGlueck2018a}.  However, under certain assumptions, sufficiently small positive perturbations can preserve the eventual positivity \cite[Theorems~4.2 and 4.9]{DanersGlueck2018a}. Leveraging results from the prequel, we provide sufficient conditions for the robustness of eventual positivity in semigroups on $L^2$-spaces, extending the existing framework to accommodate a broader class of (small) perturbations.

Recall that a linear operator $B:\dom(B)\subseteq H\to H$ on a Hilbert space $H$ is said to be $\emph{symmetric}$ if it is densely defined and
\[
    \braket{ Bx, y}_H =\braket{x, By}_H \quad \text{for all }x,y\in \dom(B).
\]
A symmetric operator is always closable \cite[p.\ 269]{Kato}.

We continue using the Banach lattice terminology recalled in the introduction. Recall that a Hilbert lattice is simply a Banach lattice whose norm is induced by an inner product. Actually, every Hilbert lattice is isometrically isomorphic to some $L^2(\Omega, \mu)$ for a measure space $(\Omega,\mu)$; see \cite[Theorem~IV.6.7]{Schaefer}. In what follows, for elements $x,y$ of a Banach lattice, we write $x\succeq y$ when there exists $c>0$ such that $x\ge cy$. Also, for $S,T\in \calL(E)$, we write $S\ge T$ if $Sx\ge Tx$ for all $x\in E$. 

Recall that the \emph{spectral bound} of an operator $A$ on a Banach space is defined by
\[
    \spb(A) \coloneqq \sup\{\re \lambda:\lambda \in \spec(A) \} \in [-\infty,\infty].
\]
 
\begin{theorem}
    \label{thm:eventual-positivity-main}
    Let $(\Omega,\mu)$ be a measure space and let $(T(t))_{t\ge 0}$ be a $C_0$-semigroup on the Banach lattice $E:=L^2(\Omega,\mu)$ with a self-adjoint and real generator $A$. Let $u\in E_+$ be such that  $(T(t))_{t\ge 0}$ is \emph{(individually) eventually strongly positive} with respect to $u$, i.e., for each $0\lneq x\in E$, there exists $t_0\ge 0$ such that $T(t)x\succeq u$ for all $t\ge t_0$.

    Let $B\in \calL(E_1,E)$ be a real and symmetric operator such that conditions~\ref{assume:it:ultracontractive}--\ref{assume:it:Vbounds} of Setting~\ref{setting:admissible} are satisfied with $X=E$ and $V=E_{u}$. Then there exists $\delta>0$ such that for all $\kappa \in (-\delta,\delta)$, there exist $\tau\ge 0$ and $\varepsilon>0$ such that the perturbed semigroup $(T_{\kappa}(t))_{t\ge 0}$ generated by $A+\kappa B$ satisfies
    \begin{equation*}
         e^{-t \spb(A+\kappa B)} T_{\kappa}(t)\ge \varepsilon \Braket{u,\argument}u \quad\text{for all } t\ge\tau.
    \end{equation*}
    In particular, $T_\kappa$ is analytic, eventually compact, and uniformly eventually strongly positive with respect to $u$.
\end{theorem}

\begin{remarks} 
    \label{rem:eventual-positivity}
    (a) For the proof of Theorem~\ref{thm:eventual-positivity-main}, we will employ a sufficient criterion for uniform eventual strong positivity of the semigroup stated in \cite[Theorem~3.1]{DanersGlueck2018b}. We point out that the statement of the aforementioned reference contains a small error:~since the proof involves replacing $A$ with $A-\spb(A)$, the correct conclusion should be
    \[
        e^{t(A-\spb(A))} \ge \varepsilon \Braket{\varphi,\argument}u \quad\text{for } t\ge t_0.
    \]
    
    (b) Even in the special case of self-adjoint semigroups, Theorem~\ref{thm:eventual-positivity-main} is more general than currently known results on perturbations of eventual positivity.
    Indeed, unlike~\cite[Theorems~4.2 and 4.9]{DanersGlueck2018a} and the results of~\cite[Section 4]{Pappu-etal}, Theorem~\ref{thm:eventual-positivity-main} does not require any (uniform) positivity of the resolvent in a neighbourhood of the spectral bound. Moreover, with the aid of Theorem~\ref{thm:eventual-positivity-main}, one can extend \cite[Example~4.10]{DanersGlueck2018a} to treat even unbounded perturbations.

    (c) The term \emph{uniform} eventual positivity refers to the fact the `time to positivity' $\tau$ can be chosen independently of the starting vector $x\in E_+\setminus\{0\}$.
    
    At first glance, the individual eventual positivity assumption on the unperturbed semigroup in Theorem~\ref{thm:eventual-positivity-main} appears surprisingly weak to guarantee uniform eventual positivity of the (un)perturbed semigroup. Nevertheless, owing to the self-adjointness assumption, individual eventual positivity is sufficient to imply its uniform counterpart, as established in \cite[Corollary~3.4]{DanersGlueck2018b}.
    
    (d) Naturally, one may ask why we are only able to obtain a perturbation result for uniform (rather than individual) eventual positivity in the present section. This is because in Theorem~\ref{thm:eigenfunction-lower-bound}, it is not possible to perturb a similar lower bound on the spectral projection. More precisely, if the spectral projection $P(0)$ corresponding to $\lambda_0$ in Theorem~\ref{thm:eigenfunction-lower-bound} satisfies $P(0)x\succeq u$  for each $x\in E_+\setminus \{0\}$, then this need not be true for the perturbed projections $P(\kappa)$ for any $\kappa$; indeed this can be checked for \cite[Example~2.4]{DanersGlueck2018a}. The reason for this is that eigenfunctions corresponding to the dual $A(\kappa)'$ may no longer be strictly positive; cf. \cite[Corollary~3.3]{DanersGlueckKennedy2016b}.

    (e) As the unperturbed semigroup in Theorem~\ref{thm:eventual-positivity-main} is self-adjoint, it is analytic. Therefore, if $\dom(A^n)\subseteq E_u$ for some $n\in \bbN$, then we can find $C\ge 1$ such that
    \[
        T(t)E\subseteq E_u \quad \text{and}\quad \norm{T(t)}_{E\to E_u} \le C t^{-n} 
    \]
    for all $t\in (0,n]$; see~\cite[Theorem~4.1]{AroraGlueck23} and its proof. In particular, the condition~\ref{assume:it:ultracontractive} in Setting~\ref{setting:admissible} is fulfilled. In concrete situations, one way to check the condition $\dom(A^n)\subseteq E_u$ is by Sobolev embeddings and elliptic regularity theory; see~\cite[Proposition 6.6]{DanersGlueckKennedy2016b} for an example involving the Dirichlet bi-Laplacian.
\end{remarks}

To facilitate the proof of Theorem~\ref{thm:eventual-positivity-main}, we establish in the subsequent lemma the spectral conditions as required for~\cite[Theorem~3.1]{DanersGlueck2018b}.

\begin{lemma}
    \label{lem:eventual-positivity-main}
    Let $(T(t))_{t\ge 0}$ be an analytic $C_0$-semigroup on a complex Banach lattice $E$ whose generator $A$ is real
    and let $u\in E_+$ be such that $(T(t))_{t\ge 0}$ is individually eventually strongly positive with respect to $u$, i.e., for each $x\in E_+ \setminus \{0\}$, there exists $t_0\ge 0$ such that $T(t)x\succeq u$ for all $t\ge t_0$.

    Let $B\in \calL(E_1,E)$ be a real operator such that the conditions~{\upshape\ref{assume:it:ultracontractive}}--{\upshape\ref{assume:it:Vbounds}} of Setting~\ref{setting:admissible} are satisfied with $X=E$ and $V=E_{u}$. If $\spb(A)\in \bbR$ is a pole of the resolvent $\Res(\argument,A)$, then
    there exists $\delta>0$ such that $\kappa \in (-\delta,\delta)$ implies that $\spb(A+\kappa B)$ is an algebraically simple eigenvalue of $A+\kappa B$ with an eigenvector $v(\kappa)\succeq u$.
\end{lemma}

\begin{proof}
    First, observe that $E_u$ is dense in $E$ due to the condition~\ref{assume:it:ultracontractive} in Setting~\ref{setting:admissible} in combination with the strong continuity of the semigroup. In particular, $u$ is a quasi-interior point of $E_+$.
    The eventual positivity assumption on the semigroup allows us to combine \cite[Theorem~5.1]{DanersGlueck2017} and \cite[Corollary~3.3]{DanersGlueckKennedy2016b} to obtain that the pole $\spb(A)$ is an algebraically simple eigenvalue of $A$ and the corresponding eigenspace contains a vector $v_0 \succeq u$. Additionally, by Example~\ref{exas:closed-unit-ball}\ref{exas:closed-unit-ball:itm:principal-ideal}, $V=E_u$ satisfies the condition~\ref{assume:bochner-V:itm:closed-ball} of Setting~\ref{setting:admissible}. 
    So, according to Theorem~\ref{thm:eigenfunction-lower-bound}(b), there exists $\delta >0$ and $r>0$ such that if $\kappa \in (-\delta,\delta)$, then $A+\kappa B$ has a unique algebraically simple eigenvalue $\lambda(\kappa)$ in $(\spb(A)-r,\spb(A)+r)$ with an eigenvector $v(\kappa)\succeq u$ and $\overline{B(\spb(A),r)}\cap \spec(A+\kappa B) = \{\lambda(\kappa)\}$.
 
    Thus, it suffices to show that $\spb(A+\kappa B)=\lambda(\kappa)$ for an arbitrary but fixed $\kappa$ with $\kappa\in(-\delta,\delta)$. As $B$ is $A$-bounded with $A$-bound strictly less than $1$ (Remark~\ref{rem:setting}(b)), the equality $A\Res(\lambda,A)=\lambda\Res(\lambda,A)-I$ allows us to obtain the estimate
    \[
        \norm{B\Res(\lambda,A)} \le \left(\modulus{\lambda}+b\right)\norm{\Res(\lambda,A)}+1
    \]
    for some $b>0$ and all $\lambda\in \rho(A)$. Therefore, employing the analyticity of the semigroup, we can ensure that $B\Res(\argument,A)$ remains bounded on 
    $H:= \{\lambda \in \bbC : \re \lambda >\spb(A)+r\}$; see \cite[Corollary~3.7.17]{ABHN}. Thus, making $\delta$ smaller if necessary, we can ensure that $\norm{\kappa B\Res(\lambda,A)}<1$ for all  $\lambda \in H$.
    This, along with the equality
    \[
        \lambda - (A+\kappa B) = (I - \kappa B\Res(\lambda,A))(\lambda-A),
    \]
    yields $H\subseteq \rho(A+\kappa B)$
    and in turn, $\spb(A+\kappa B)\le \spb(A)+r$. As $\overline{B(\spb(A),r)}$ contains only one spectral value of $A+\kappa B$, we get  $\spb(A+\kappa B)=\lambda(\kappa)$.
\end{proof}

\begin{proof}[Proof of Theorem~\ref{thm:eventual-positivity-main}]
    As noted in Lemma~\ref{lem:eventual-positivity-main}, the vector $u$ is a quasi-interior point of $E_+$. Also,
    by Example~\ref{exas:closed-unit-ball}\ref{exas:closed-unit-ball:itm:principal-ideal}, $V=E_u$ satisfies the condition~\ref{assume:bochner-V:itm:closed-ball} of Setting~\ref{setting:admissible}.
    Theorem~\ref{thm:perturb-smoothing} now implies that for each $\kappa \in (-1,1)$, the operator $A+\kappa B$ generates a $C_0$-semigroup that satisfies $T_{\kappa}(t)E\subseteq E_u$ for all $t>0$; cf. Remark~\ref{rmk:perturb-with-parameter}. 

    Reflexivity of $E$ and the assumption $T(1)E\subseteq E_u$ allow us to apply \cite[Corollary~2.5]{DanersGlueck2017} to conclude that all spectral values of $A$ are poles of the resolvent. Now,
    as $A$ is self-adjoint and $B$ is symmetric and $A$-bounded with $A$-bound strictly less than $1$ (Remark~\ref{rem:setting}(b)), so $A+\kappa B$ is self-adjoint by \cite[Theorem~V-4.3]{Kato}. Thus, $\spb(A+\kappa B)\in \bbR$ is a dominant spectral value of $A+\kappa B$ for all $\kappa \in (-1,1)$. Applying Lemma~\ref{lem:eventual-positivity-main} yields $\delta>0$ such that $\spb(A+\kappa B)$ is an algebraically simple eigenvalue of $A+\kappa B$ with an eigenvector $v(\kappa)\succeq u$ whenever $\kappa \in (-\delta,\delta)$.

    Next, since $u$ is a quasi-interior point of $E_+$, $u$ also denotes a strictly positive functional on $E$ by a characterization of quasi-interior points; see, for instance, \cite[Theorem~II.6.3]{Schaefer}. Thus, thanks to the self-adjointness of $A+\kappa B$ (and in turn the perturbed semigroup), the above conditions are also fulfilled for the dual. The assertions of uniform eventual strong positivity and eventual compactness are therefore a consequence of \cite[Theorem~3.1]{DanersGlueck2018b}; cf.\ Remark~\ref{rem:eventual-positivity}(a). The analyticity of the perturbed semigroup is simply due to the self-adjointness of $A+\kappa B$.
\end{proof}

In the proof of Theorem~\ref{thm:eventual-positivity-main}, we have implicitly used the fact that if $A$ is a densely defined real operator on a Hilbert lattice $E$, then its Banach space dual $A'$ coincides with the Hilbert space adjoint $A^*$; see~\cite[p.\ 10]{DanersGlueck2018b} for the details.

\subsubsection*{Nonlocal Robin boundary conditions}

We conclude the article with an application of Theorem~\ref{thm:eventual-positivity-main} to a second-order elliptic operator with non-standard boundary conditions.

In the setting of second-order elliptic operators with classical boundary conditions, as in Proposition~\ref{prop:heat-semigroups} for instance, the semigroup is positive and ultracontractive. The positivity can be lost if we consider non-standard boundary conditions. As a simple example of \emph{nonlocal} Robin boundary conditions on a bounded Lipschitz domain $\Omega\subset\bbR^n$, the multiplication by $\beta$ in the form~\eqref{eq:form-robin} can be replaced by a general bounded linear operator $N$ acting on $L^2(\partial\Omega)$:
\begin{equation}
    \label{eq:nonlocal-robin}
    \fra_B[u,v] \coloneqq \int_\Omega A\nabla u\cdot\overline{\nabla v}\dd x + \int_{\partial\Omega} (Nu)\overline{v}\dd\sigma, \quad u,v\in H^1(\Omega).
\end{equation}
This is not a new consideration, e.g.\ see~\cite{GesztesyMitrea2009, GesztesyMitreaNichols2014, AKK18} and~\cite[Section 6]{DanersGlueckKennedy2016b}; however, a systematic investigation of positivity properties (or lack thereof) of semigroups generated by such operators was initiated in the recent article~\cite{GlueckMui2024}. 
More complicated non-local terms, encoding dynamical boundary conditions of Wentzell-Robin type, were studied recently in~\cite{KunzeMuiPloss}.

Suppose $N\in \calL(L^2(\partial\Omega))$ is such that $N$ extrapolates to a bounded operator on $L^1(\partial\Omega)$ and $L^\infty(\partial\Omega)$ -- in particular, these conditions include the classical case where $N$ is multiplication by a bounded real-valued function $\beta$.
Then it is proved in~\cite[Theorem 3.2]{GlueckMui2024} that the semigroup $(T(t))_{t\ge 0}$ generated by $-L$, where $L$ is the elliptic operator arising from the form~\eqref{eq:nonlocal-robin}, satisfies
\begin{equation}
\label{eq:nonlocal-robin-ultrac}
    \|T(t)\|_{2\to\infty} \le ct^{-\mu/4}, \qquad t\in (0,1]
\end{equation}
for suitable constants $c,\mu>0$ depending only on the dimension $n$. Thus the ultracontractivity condition is satisfied. Sufficient conditions for (uniform) eventual positivity with respect to the positive principal eigenfunction of $L$ are given in~\cite[Theorem~4.4 and Theorem~5.5(ii)]{GlueckMui2024}.

Theorem~\ref{thm:eventual-positivity-main} shows that the eventual positivity of the semigroup can be preserved even with non-positive perturbations in the interior.
\begin{example}
    {
    Let $L:\dom(L)\subset L^2(\Omega)\to L^2(\Omega)$ be the elliptic operator associated to the form~\eqref{eq:nonlocal-robin}, where we assume in addition that the coefficient matrix $A(x)$ is symmetric for each $x\in\Omega$. Let $0 \ne v\in L^\infty(\partial\Omega;\bbR)$ satisfy $\int_{\partial\Omega} v \dd\sigma = 0$. Define the operator $N\in\calL(L^2(\partial\Omega))$ by
    \begin{equation*}
        Nf \coloneqq (v\otimes v)(f) \coloneqq \langle v,f \rangle_{L^2(\partial\Omega)}v, \qquad f\in L^2(\partial\Omega).
    \end{equation*}
    Then $N$ is real and self-adjoint, since $v$ is real-valued. Consequently the form $\mathfrak{a}_B$ is symmetric, so that $L$ and the associated semigroup $(T(t))_{t\ge 0}$ generated by $-L$ are self-adjoint.

    Clearly, $N$ extrapolates to a bounded operator on $L^1(\partial\Omega)$ and $L^\infty(\partial\Omega)$, satisfies $N\one_{\partial\Omega} = 0$, and $\langle Nf,f \rangle_{L^2(\partial\Omega)} \ge 0$ for all $f\in L^2(\partial\Omega)$.
    Thus~\cite[Theorem 4.4]{GlueckMui2024} applies to show that $(T(t))_{t\ge 0}$ is eventually strongly positive with respect to $\one_\Omega$, i.e., there exists $\tau\ge 0$ such that $T(t)f \succeq \one_{\Omega}$ for all non-zero $f\in L^2(\Omega)_+$ and $t\ge \tau$. However, it is shown in~\cite[Example 4.5(a)]{GlueckMui2024}} that $(T(t))_{t\ge 0}$ is \emph{not} positive.

    Theorem~\ref{thm:eventual-positivity-main} shows that some non-positive perturbations in $\Omega$ can preserve the eventual positivity of the semigroup. A simple example is given by a kernel operator
    \begin{equation}
    \label{eq:B-kernel-op}
        Bu \coloneqq \int_\Omega k(\argument,y)u(y) \dd y, \qquad u\in L^2(\Omega)
    \end{equation}
    where $k\in L^\infty(\Omega\times\Omega;\bbR)$ is symmetric, i.e.,  $k(x,y)=k(y,x)$ for all $x,y\in\Omega$, and no assumption on the sign of $k$ is needed. Then clearly $B$ is a real and symmetric operator, and is a bounded operator on $L^2(\Omega)$ that leaves $L^\infty(\Omega)$ invariant. Therefore, by~\eqref{eq:nonlocal-robin-ultrac} and Theorem~\ref{thm:ultracontractivity-introduction} applied to $X=L^2(\Omega)$ and $V=L^\infty(\Omega)$, we deduce that the perturbed semigroup $(S(t))_{t\ge 0}$ generated by $-L+B$ satisfies the ultracontractivity estimate
    \begin{equation*}
        \|S(t)\|_{2\to\infty} \le \widetilde{C}t^{-\mu/4}, \qquad t\in (0,1].
    \end{equation*}
    Furthermore, by Theorem~\ref{thm:eventual-positivity-main} and Remark~\ref{rem:holomorphic-family}, we deduce that if $\|B\|_{\calL(L^2(\Omega))} \le \delta$ for sufficiently small $\delta > 0$, then there exist $\varepsilon > 0$ and $\tau\ge 0$ such that the perturbed semigroup satisfies
    \begin{equation*}
        e^{-t\spb(-L+B)}S(t) \ge \varepsilon \langle\one_\Omega, \argument\rangle \one_\Omega \quad\text{for all }t\ge\tau.
    \end{equation*}
\end{example}

\begin{remark}
    Based on the results of~\cite{CaffarelliStinga2016} and the `nice' behaviour of heat semigroups on the scale of Sobolev spaces $W^{2s,p}(\Omega)$, it seems plausible that our abstract results could admit perturbations by operators of the form~\eqref{eq:B-kernel-op} where the kernel function $k$ has certain types of singularities, such as those encountered in the integral kernels of fractional Laplacians; see e.g.\ \cite[Theorem 2.4]{CaffarelliStinga2016}. However, we refrain from an in-depth investigation of such questions in this paper.
\end{remark}

\section*{Acknowledgements}
We thank Jochen Gl\"{u}ck for many valuable discussions, and in particular for his advice regarding Proposition~\ref{prop:bochner-V} and the proof of Theorem~\ref{thm:perturb-smoothing}. We are also grateful to Hendrik Vogt for pointing out an oversight in an earlier version of Lemma~\ref{lem:V-bounded}.

The work was primarily completed while the first author was funded by the Deutsche Forschungsgemeinschaft (DFG, German Research Foundation) Project 523942381. The second author is funded by the DFG Project 515394002. Parts of the work were done during the first author's visits to the University of Wuppertal and the second author's visits to the University of Twente.

\bibliographystyle{plainurl}
\bibliography{literature}

\appendix

\section{Perturbations of the spectrum}
    \label{appendix:spectrum}

In this appendix, we briefly review how relatively bounded perturbations preserve the spectral properties of closed operators.
To begin, we recall Kato's notion of \emph{generalised convergence} of closed operators. For closed subspaces $M, N$ of a Banach space $X$, define
    \begin{equation}
        \delta(M,N) := \sup_{u\in S_M}\dist(u,N), \quad \hat{\delta}(M,N) := \max\{\delta(M,N), \delta(N,M)\}
    \end{equation}
    with the convention that $\delta(\{0\},N) := 0$ for any subspace $N$; here $S_M$ denotes the unit sphere of $M$ in $X$. The quantity $\hat{\delta}$ is called the \emph{gap between $M$ and $N$}.
    Now let $X,Y$ be Banach spaces and suppose $A,B$ are closed operators from $X$ to $Y$. 
    The \emph{gap between $A$ and $B$} is defined as the gap between their respective graphs $\Gamma(A)$ and  $\Gamma(B)$, i.e.,
    \begin{equation}
    \label{eq:gap-AB}
        \hat{\delta}(A,B) := \hat{\delta}(\Gamma(A),\Gamma(B)).
    \end{equation}
    Note that~\eqref{eq:gap-AB} is well-defined since $\Gamma(A), \Gamma(B)$ are closed subspaces of the Banach space $X\times Y$.
    If $B$ is a bounded operator on $X$, then by~\cite[Theorem IV-2.14]{Kato}
    \begin{equation*}
        \hat{\delta}(A+B,A) \le \|B\|_{\calL(X)}.
    \end{equation*}
    With this definition at hand, we identify operators with their graphs and say that a sequence $(A_n)_{n\ge 1}\subset X\times Y$ of closed operators \emph{converges in the generalised sense} to a closed operator $A$ if
    \begin{equation*}
        \hat{\delta}(A_n, A) \to 0 \quad\text{as } n\to\infty.
    \end{equation*}
Many fundamental properties of the gap functional were derived in~\cite[Section IV-2]{Kato}. In particular, a characterisation of generalised convergence of closed operators via convergence of resolvent operators can be found in \cite[Theorem IV-2.25]{Kato}.
    
\begin{remark}
    Observe that $\delta(M,N)$ can be defined equivalently as
    \begin{equation*}
        \delta(M,N) := \inf\{c\ge 0 \suchthat \dist(u,N) \le c\|u\| \,\text{for all }u\in M\}.
    \end{equation*}     
    The definition of $\hat{\delta}$ resembles the Hausdorff distance for subsets of a metric space. However, $\hat{\delta}$ is not a proper distance function, since the triangle inequality need not be satisfied. One can in fact construct a distance function $\hat{d}$ on the set of all closed subspaces of a Banach space that is equivalent to $\hat{\delta}$ in the sense that
    \begin{equation*}
        \hat{\delta}(M,N) \le \hat{d}(M,N) \le 2\hat{\delta}(M,N)
    \end{equation*}
    always. Even so, the definition of the gap is less cumbersome and suitable for most purposes. See~\cite[Section IV-2.1]{Kato} for further details.
\end{remark}

The following proposition is a typical result on perturbations of the spectrum and is a natural generalisation of~\cite[Lemma 3.3]{DanersGlueck2018a}. The conclusions are of a standard nature -- indeed, they are consequences of more general arguments developed in the classic monograph~\cite{Kato} of Kato. However, since we do not require Kato's machinery in full generality, we  provide a statement that is adapted to our purposes.

    Recall that for linear operators $A,B : X\to Y$ between Banach spaces $X$ and $Y$, $B$ is said to be \emph{relatively bounded with respect to} $A$, or simply $A$-\emph{bounded}, if $\dom(A)\subseteq \dom(B)\subseteq X$ and there exist constants $a,b\ge 0$ such that
    \begin{equation}
    \label{eq:T-bound}
        \|Bu\|_Y \le a\|u\|_X + b\|Au\|_Y \quad\text{for all } u\in \dom(A).
    \end{equation}
    The infimum of all $b\ge 0$ for which inequality~\eqref{eq:T-bound} holds is called the $A$-\emph{bound} of $B$. If $B$ is $A$-bounded with $A$-bound $<1$ and $A$ is closed, then by~\cite[Theorem IV-1.1]{Kato}, the sum $A+B$ is a closed operator with $\dom(A+B)=\dom(A)$.

\begin{proposition}
    \label{prop:spectral-perturb}
    Let $X$ be a complex Banach space and let $A:\dom(A)\subseteq X\to X$ be a closed operator. Assume that $\lambda_0\in\spec(A)$ is a pole of the resolvent $\Res(\argument,A)$ with corresponding spectral projection $P_0$. Let $r>0$ be such that $\overline{B(\lambda_0,r)}\cap\spec(A) = \{\lambda_0\}$. Moreover, for every $A$-bounded operator $B$, define
    \[
        \beta_B \coloneqq \sup_{\modulus{\lambda-\lambda_0}=r} \norm{B \Res(\lambda,A)}.
    \]

    Then the following assertions hold:
    \begin{enumerate}[\upshape(a)]
        \item The convergence $A+B\to A$ is true in the generalised sense as $\beta_B\to 0$.
        \item There exists $\beta_0>0$ such that whenever $\beta_B \in (0,\beta_0)$, the operator $A+B$ has a unique spectral value $\lambda_B$ in $\overline{B(\lambda_0,r)}$ that is a pole of the resolvent $\Res(\argument,A+B)$ with the corresponding spectral projection $P_B$ satisfying $\dim\rg P_B=\dim\rg P_0$. 
        \item We have $\norm{P_B-P_0}\to 0$ as $\beta_B\to 0$.
    \end{enumerate}
    In addition, if $X$ is a complex Banach lattice, $\lambda_0\in\bbR$, and the operators $A,B$ are real, then $\lambda_B\in\bbR$.
\end{proposition}

\begin{proof}
    According to \cite[Theorem~IV-3.16]{Kato}, the assertions~(b) and~(c) follow from~(a).
    To see~(a),
    first of all, note that the finiteness of $\beta_B$ is guaranteed by the $A$-boundedness of $B$. Next, let $C_r$ denote the positively oriented circle of radius $r>0$ centred at $\lambda_0$. Since $\lambda I - (A+B) = (I - B\Res(\lambda,A))(\lambda I - A)$, so whenever $\beta_B<1$, a Neumann series expansion yields that $\lambda I - (A+B)$ is invertible with
    \[
        \Res(\lambda, A+B) = \Res(\lambda,A)[I-B\Res(\lambda,A)]^{-1} = \Res(\lambda,A) \sum_{k=0}^\infty [B\Res(\lambda,A)]^k
    \]
    and in turn,
    \[
        \norm{\Res(\lambda,A+B)-\Res(\lambda,A)} \le \frac{\beta_B}{1-\beta_B} \norm{\Res(\lambda,A)}
    \]
    for all $\lambda\in C_r$. Fixing $\lambda \in C_r$, it follows that $\norm{\Res(\lambda,A+B)-\Res(\lambda,A)}\to 0$ as $\beta_B\to 0$. 
    Therefore, $A+B\to A$ in the generalised sense as $\beta_B\to 0$, by the characterisation of generalised convergence in \cite[Theorem IV-2.25]{Kato}.

    Finally, assume that $\lambda_0\in \bbR$ and the operators $A,B$ are real on the complex Banach lattice $X$. If $\lambda_B\notin\bbR$, then $\overline{\lambda_B}$ is another spectral value of $A+B$ in the disk $B(\lambda_0,r)$, which contradicts assertion (b). 
\end{proof}

\begin{remark}
    Observe that in Proposition~\ref{prop:spectral-perturb}, choosing $r>0$ sufficiently small yields a spectral value of $A+B$ arbitrarily close to $\lambda_0$.
\end{remark}

\section{Bochner integrals}
    \label{appendix:bochner}

In this appendix, we give a technical result about Bochner integration which might be well known to experts, but we could not locate an explicit reference.
\begin{proposition}
    \label{prop:bochner-V}
    Let $X$ be a Banach space, and let $f: (0,\tau) \to X$ be a strongly measurable function for some fixed $\tau>0$. Assume that $V$ is a Banach space embedding continuously into $X$ such that
    $f(t) \in V$ a.e. $t \in (0,\tau)$ and there exists a function $\beta\in L^1((0,\tau),\bbR_+)$ such that $\norm{f(t)}_V \le \beta(t)$ a.e. $t \in (0,\tau)$.

    Then $f$ is Bochner integrable with values in $X$.
    In addition, suppose that at least one of the following conditions hold:
    \begin{enumerate}[\upshape(a)]
        \item\label{lem:bochner-V:itm:separable} 
        The Banach space $V$ is separable. 

        \item\label{lem:bochner-V:itm:closed-ball} 
        The closed unit ball of $V$ is closed in $X$.
    \end{enumerate}
    Then
    \begin{equation}
        \label{eq:bochner-V-estimate}
        \int_0^{\tau} f(t) \dd t \in V \quad\text{and}\quad \norm{\int_0^{\tau} f(t) \dd t}_V \le \int_0^{\tau} \beta(t) \dd t.
    \end{equation}
\end{proposition}

Condition~\ref{lem:bochner-V:itm:separable} or~\ref{lem:bochner-V:itm:closed-ball} is satisfied in many concrete situations; see Example~\ref{exas:closed-unit-ball} below. Hence it is tempting to conjecture that the conditions are redundant. However, this is false, as can be seen from a counterexample due to David Gao on MathOverflow~\cite{MO-Gao}.

\begin{proof}[Proof of Proposition~\ref{prop:bochner-V}]
    As $V$ embeds continuously into $X$, there exists a constant $C > 0$ such that $\norm{f(t)}_X \le C \norm{f(t)}_V \le C \beta(t)$ for almost all $t \in (0,\tau)$. Therefore, we infer from the integrability of $\beta$ and measurability of $f$ that the latter is Bochner integrable with values in $X$. 

    Now we show that each of the assumptions~\ref{lem:bochner-V:itm:separable} and~\ref{lem:bochner-V:itm:closed-ball} implies~\eqref{eq:bochner-V-estimate}.

    \ref{lem:bochner-V:itm:separable}
    For every functional $x' \in X'$, the scalar-valued function $(x'|_V) \circ f = x' \circ f$ is measurable. 
    As the set $\{x'|_V : \, x' \in X'\}$ is a separating subset of $V'$ and $V$ is separable, $f$ is strongly measurable with values in $V$ \cite[Corollary~1.1.3]{ABHN}. 
    The norm estimate on $f$ now ensures that $f$ is Bochner integrable as a $V$-valued function \cite[Theorem~1.1.4]{ABHN}.
    As $V$ embeds continuously into $X$, the Bochner integral of $f$ in $V$ coincides with its Bochner integral in $X$, which proves~\eqref{eq:bochner-V-estimate}.

    \ref{lem:bochner-V:itm:closed-ball} We begin with a general result which follows immediately from~\cite[Proposition 1.2.12]{HytoenenvanNeervanVeraarWeis2016a}:~Let $E$ be any Banach space, $(\Omega,\mu)$ be a probability space, and $\xi:\Omega\to E$ be a Bochner-integrable random variable. If $K\subset E$ is a closed convex subset such that $\xi\in K$ almost surely, then the expectation of $\xi$ also lies in $K$, i.e.,
    \begin{equation*}
        \int_\Omega \xi \dd\mu \in K.
    \end{equation*}
    
    Now consider the set $J \coloneqq \{ t\in (0,\tau) ~\colon \beta(t) > 0 \}$, which we assume has positive Lebesgue measure (otherwise $f(t)=0$ for a.e.\ $t$ and we are done). Then $g\coloneqq\frac{f}{\beta}$
    is well-defined on $J$. Set
    \[
        M \coloneqq \int_{J}\beta(t)\dd t = \int_0^{\tau} \beta(t)\dd t\quad\text{and}\quad\dd\mu(t) \coloneqq M^{-1}\beta(t)\dd t.
    \]
    Clearly, $\mu$ is a probability measure on $J$. The assumptions on $f$ imply that $g$ lies in the closed unit ball of $V$, which is a convex set and closed in $X$ by assumption. By the result quoted above, it thus follows that
    \[
        M^{-1}\int_{(0,\tau)} f(t)\dd t = \int_J g(t)\dd\mu(t)
    \]
    lies in the closed unit ball of $V$ and in turn~\eqref{eq:bochner-V-estimate} holds.
\end{proof}

Recall that a \emph{KB-space} is a Banach lattice in which the norm-bounded increasing sequences are norm-convergent. Standard examples of KB-spaces include all $L^p$-spaces with $p\in [1,\infty)$ and in fact, all reflexive Banach lattices. A subspace $I$ of a Banach lattice $E$ is said to be \emph{lattice ideal} if for each $x,y\in E$, the conditions $0\le x\le y$ and $y\in I$ imply that $x\in I$.

\begin{examples}
    \label{exas:closed-unit-ball}
    \leavevmode
    We collect some examples of Banach spaces $X$ and $V$ such that $V$ embeds continuously into $X$ and the closed unit ball of $V$ is closed in $X$:
    \begin{enumerate}[label=(\alph*)]
        \item\label{exas:closed-unit-ball:itm:L_p} 
        For every finite measure space $(\Omega,\mu)$ and $1 \le p \le \infty$, it follows from part~\ref{exas:closed-unit-ball:itm:principal-ideal} below that the closed unit ball of $V := L^\infty(\Omega,\mu)$ is closed in $X := L^p(\Omega,\mu)$. 

        \item\label{exas:closed-unit-ball:itm:reflexive}
        If $V$ is reflexive, then its closed unit ball is closed in $X$. Indeed, the reflexivity of $V$ guarantees that its closed unit ball is weakly compact in $V$. Continuity of the embedding $V \hookrightarrow X$ hence guarantees that the closed unit ball of $V$ is also weakly compact in $X$. In turn, it is (weakly) closed in $X$.
        
        \item\label{exas:closed-unit-ball:itm:Lip} 
        Let $V\coloneqq \Lip(M)$ be the space of all Lipschitz continuous functions on functions on a bounded metric space $M$ endowed with the norm $\norm{\argument}\coloneqq \Lip(\argument)+\norm{\argument}_\infty$, where $\Lip(f)$ denotes the Lipschitz constant of a function $f$. 
        If $(f_n)$ is a sequence of uniformly bounded Lipschitz continuous functions in $V$ that converges uniformly to a bounded continuous function $f$, then $f$ is also Lipschitz with the same Lipschitz constant. Hence the closed unit ball of $V$ is closed in $X\coloneqq \big(\Cont_{\mathrm{b}}(M),\norm{\argument}_{\infty}\big)$.  

        \item\label{exas:closed-unit-ball:itm:principal-ideal}
        If $X$ is a Banach lattice and $u \in X_+$, then the closed unit ball of the principal ideal $V\coloneqq X_u$ (endowed with the gauge norm) is closed in $X$; see Section~\ref{sec:notation} for the relevant definitions.

        To see this, let $(v_n)$ be a sequence in the closed unit ball of $V$ converging in $X$ to a vector $x\in X$. By definition of the gauge norm, $\modulus{v_n}\le u$ for all $n\in\bbN$ and by continuity of lattice operations $(\modulus{v_n})$ converges to $\modulus{x}$ in $X$. In particular, $\modulus{x}\le u$ and so $x\in V$, as desired.

        \item\label{exas:closed-unit-ball:itm:ideal-kb}
        If $X$ is a Banach lattice, $V$ is a KB-space, and $V$ is a lattice ideal in $X$, then the closed unit ball of $V$ is closed in $X$. 

        This can be argued as follows:~let $(v_n)$ be a sequence in the closed unit ball $B_V$ of $V$ that $X$-converges to a vector $x \in X$. Without loss of generality, we assume that $(v_n)$ is a sequence in the positive unit ball $(B_V)_+$ of $V$ and that $x \in X_+$. 
        Indeed, the sequences of positive and negative parts $(v_n^+)$ and $(v_n^-)$ lie in $(B_V)_+$ and converge in $X$ to $x^+$ and $x^-$, respectively (note that it does not matter whether we apply the lattice operators to the $v_n$ in $V$ or in $X$ since $V$ is an ideal and hence a sublattice of $X$). We need to show that $x\in (B_V)_+$.
        
        In fact, we can also assume that $v_n \le x$ for each index $n$. This is because $V$ is an ideal in $X$, and so the sequence $(v_n \land x)$ is also located in $(B_V)_+$. From the continuity of the lattice operations, we have $v_n \wedge x \to x$ in $X$ as well. 
        
        Replace $(v_n)$ with a subsequence to achieve that $\sum_{n=1}^\infty \norm{x-v_n}_X < \infty$ and then define $u_n := \bigwedge_{k=n}^\infty v_k$ for each $n \in \bbN$, where the infimum exists in $V$ and thus in $X$ (as $V$ is an ideal in $X$). 
        Then $u_n \le x$ and $u_n \in (B_V)_+$ for each $n$. 
        The sequence $(u_n)$ is increasing and $X$-converges to $x$ since
        \[
            \quad \qquad 
            0 
            \le 
            x - u_n 
            = 
            x - \bigwedge_{k=n}^\infty v_k 
            = 
            \bigvee_{k=n}^\infty (x-v_k) 
            \le 
            \sum_{k=n}^\infty (x-v_k)
            \to 
            0
        \]
        for $n \to \infty$. 
        On the other hand, as $V$ is a KB-space and $(u_n)$ is an increasing sequence in its positive unit ball $(B_V)_+$, it follows that $(u_n)$ converges in $V$ to a vector $u \in (B_V)_+$. 
        As $V$ embeds continuously into $X$, it follows that $x = u \in (B_V)_+$.
    \end{enumerate}
\end{examples}

\begin{remark}
    In~Example~\ref{exas:closed-unit-ball}\ref{exas:closed-unit-ball:itm:ideal-kb}, one cannot drop the assumption that $V$ is a KB-space:~As a counterexample, let $V = c_0$ and let $X = \ell^1(\bbN, \mu)$, where $\mu$ is the probability measure on $\bbN$ given by $\mu(A) = \sum_{k \in A} \frac{1}{2^k}$ for each $A \subseteq \bbN$. 
        Then $V$ is a lattice ideal in $X$  but $V$ is not a KB-space.  
        The sequence $(\one_{\{1,\dots,n\}})$ in the positive unit ball of $V$ converges in $X$ to the constant sequence $\one$, but $\one \not\in V$.
\end{remark}

\end{document}